\documentclass[11pt]{article}

\usepackage{graphicx,epsfig,amssymb,amsbsy,amsmath,amsthm,cleveref,color,epstopdf,mathtools,amsfonts,amstext,subfig,pstricks,extarrows}

\newcommand{\bb}[1]{\left({#1}\right)}					
\newcommand{\cc}[1]{\left\{#1\right\}}					

\newcommand{\sfrac}[2]{\mbox{$\frac{#1}{#2}$}}	
\newcommand{\hf}{\mbox{$\frac12$}}

\newcommand{\eps}{\varepsilon}

\newtheorem{theorem}{Theorem}
\newtheorem{lemma}[theorem]{Lemma}

\newtheorem{proposition}[theorem]{Proposition}
\newtheorem{definition}{Definition}

\newtheorem{remark}{Remark}

\newcommand{\sign}{\operatorname{sign}}

\newcommand{\cosec}{\operatorname{cosec}}

\newcommand{\crefs}[2]{\cref{#1} to \cref{#2}}

\crefname{theorem}{Theorem}{Theorems}
\crefname{lemma}{Lemma}{Lemmas}
\crefname{proposition}{Proposition}{Propositions}
\crefname{section}{Section}{Sections}
\crefname{figure}{Figure}{Figures}
\crefname{equation}{}{}
\Crefname{equation}{}{}

\def\eps{\varepsilon}

\begin{document}


\title{Ageing of an oscillator due to frequency switching}

\author{Carles~Bonet$^*$, Mike R. Jeffrey$^\dagger$, Pau~Mart\'in$^*$, Josep~M.~Olm$^\#$\\\\
\small $^*$Department of Mathematics, Universitat Polit\`{e}cnica de Catalunya, Spain \\
\small $^\#$Department of Mathematics \& Institute of Industrial and Control Engineering, \\
\small Universitat Polit\`{e}cnica de Catalunya, Spain\\
\small $^\dagger$Department of Engineering Mathematics, University of Bristol, UK
}



\maketitle

\begin{abstract}
If an oscillator is driven by a force that switches between two frequencies, the dynamics it exhibits depends on the precise manner of switching. Here we take a one-dimensional oscillator and consider scenarios in which switching occurs: (i) between two driving forces which have different frequencies, or (ii) as a single forcing whose frequency switches between two values. The difference is subtle, but entirely changes the long term behaviour, and concerns whether the switch can be expressed linearly or nonlinearly in terms of a discontinuous quantity (such as a sign or Heaviside step function that represents the switch between frequencies). In scenario (i) the oscillator has a stable periodic orbit, and the system 
can be described as a {\it Filippov system}. In scenario (ii) the oscillator exhibits {\it hidden dynamics}, which lies outside the theory of Filippov's systems, and causes the system to be increasingly (as time passes) dominated by sliding along the frequency-switching threshold, and in particular if periodic orbits do exist, they too exhibit sliding. We show that the behaviour persists, at least asymptotically, if the systems are regularized (i.e. if the switch is modelled as a smooth transition in the manner of (i) or (ii)). 
\end{abstract}



\section{Introduction}

The theory of piecewise-smooth dynamical systems enables us to study how systems behave at `switching thresholds' where they suffer discontinuities. Filippov showed in \cite{f88} how to study such systems by forming a differential inclusion across the discontinuity, creating a set-valued problem from which one may select a range of possible solutions. Interest has grown in whether these different possible solutions have practical relevance, whether they behave differently, and what theoretical or practical criteria can be drawn to choose between them, particularly in light of their growing range of applications (see e.g. \cite{bc08,j18book,physDspecial} for broad overviews).

To illustrate the issues that arise in looking beyond so-called {\it Filippov systems}, a simple oscillator was proposed in \cite{j15hidden,j18book} in which the appropriate choice of solutions was unclear. The model takes the form of a second order oscillator switching between two frequencies,
\begin{align}\label{3sys}
\dot y=-a y-z-\sin(\pi\omega t)\;,\qquad \dot z=y\;,
\end{align}
where the dot denotes the derivative with respect to $t$. The sinusoidal forcing switches between two frequencies, some $\omega=\omega_+$ for $ y>0$ and $\omega=\omega_-$ for $ y<0$. Solving this system either numerically or analytically proves challenging, and centers around how the discontinuity is handled.  
In particular we may express the forcing as switching between two sinusoids with different frequencies,
\begin{subequations}\label{3}
\begin{align}\label{3lin}
\sin(\pi\omega t)=\hf(1+\lambda)\sin(\pi\omega_+t)+\hf(1-\lambda)\sin(\pi\omega_-t)\;,
\end{align}
or as one sinusoid that switches between two frequencies,
\begin{align}\label{3non}
\sin(\pi\omega t)=\sin\bb{\cc{(1+\lambda)\omega_++(1-\lambda)\omega_- }\frac{\pi t}2}\;,\qquad
\end{align}
\end{subequations}
both in terms of a discontinuous quantity $\lambda=\sign( y)$, called a {\it switching multiplier}. The $\sign$ function takes values $+1$ for $ y>0$, $-1$ for $y<0$, and $\lambda\in(-1,+1)$ for $ y=0$. The two models \cref{3lin} and \cref{3non} are therefore equivalent for $y\neq0$, but, as observed in \cite{j15hidden}, they differ crucially in the dynamics they generate at $y=0$, and this utterly changes their global behaviour. 

Why this happens was left as an open challenge in \cite{j15hidden}. 
The dynamics of the oscillator is complex and highly sensitive, and is easily mis-calculated in numerical simulations. This means that the difference between the models \cref{3lin} and \cref{3non} may not easily reveal itself. Our aim here is to begin investigating the qualitative dynamical features that organize the true behaviour of the system, and to see why they make simulating it so challenging. 

Oscillators with discontinuities have been an important application in the general development of piecewise-smooth dynamics, principally in dry-friction oscillators \cite{shaw85,popp98,kp08,awrej05}, impact oscillators \cite{sh83,popp97,nord2000,wiercigroch10}, and industrial applications abound in problems such as valves \cite{valve12}, drills \cite{wiercigroch17}, and braking \cite{singh08,hoffman12}. Approaches to study these involve either hybrid systems or complimentarity constraints \cite{l04,b99,bc08}, or Filippov's convex approach \cite{bc08,kp08,krg03}. 

Filippov's approach consists of forming a differential inclusion across a discontinuity, thus creating a convex set of possible trajectories across it, yielding solutions that can cross through or slide along switching thresholds in a largely unique manner (up to certain singularities, see e.g. \cite{jc12}), and an extensive theory of their existence, uniqueness, and nonlinear dynamics now exists (see e.g. \cite{f88,bc08,krg03}). 
For systems with multiple switches, or involving nonlinear functions of a discontinuous term, something beyond Filippov's analysis is required to define a system and its solutions, as developed in \cite{j13error,j18book}. 

Solutions through a discontinuity are inescapably non-unique. Although this  is a standard result of dynamical systems theory (see e.g. \cite{f88,jm82,m07,j15hidden}), little is currently understood about what physical meaning, if indeed any, the infinity of possible solutions might have. The system \cref{3sys} was conceived to illustrate a simple situation in which the Filippov's convex method might not be the most natural, and in which this would have non-trivial consequences for the dynamics.  

To gain insight into the complex dyamics of the system \cref{3sys}-\cref{3}, we propose here to first study a simplified first order oscillator, 
\begin{align}\label{1d} 
\dot y&=-ay-\sin(\pi\omega t)\;,
\end{align}
which exhibits most of the key features that make \cref{3sys} so challenging. The original system \cref{3sys} will then be analysed in follow-up work. 

As well as being lower dimensional, the first order system \cref{1d} has the advantage of representing a simple electronic RL circuit, as shown in \cref{fig:RL}. A circuit with current $i$, driven through a resistance $R$ and inductance $L$ by an alternating current voltage source $V(t)$, satisfies the equation $V(t)=L\frac{di}{dt}+Ri$, or 
\begin{align}\label{RL}
\frac{di}{dt}=-\frac{R}{L}i+\frac{V(t)}{L}\;.
\end{align}
By letting $\frac{R}{L}=a$ and $V(t)=-L\sin (\pi\omega t)$, we obtain \cref{1d}. 
\begin{figure}[t]\centering
\includegraphics[width=0.4\textwidth]{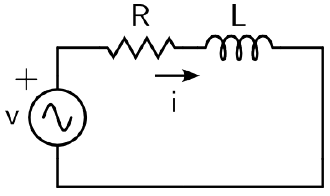}
\caption{A resistor-inductor circuit with an alternating current voltage source. }\label{fig:RL}
\end{figure}
If the frequency of the alternating current source is then switched between values $\omega_\pm$ whenever the current reverses direction (when $i$ passes through zero), the question of whether this is modelled by \cref{3lin} or \cref{3non} becomes one of the precise control feedback between the current and the frequency switch. A linear controller switching between $\omega=\omega_\pm$ with a small hysteresis will be well modelled by \cref{3lin}, while a variable frequency voltage source should be well modelled by \cref{3non}; the theoretical distinction was proven in \cite{j16hyst}. We shall see that \cref{3lin} leads to stable oscillations, while \cref{3non} causes the system to age and become unable to alternate its current, i.e. $i$ becomes fixed in $i\le0$. 

In this paper we shall show that the `linear switching' system, where the discontinuity is governed by \cref{3lin}, has relatively simple dynamics dominated by a stable period 4 orbit. For small damping parameter $a$ the period 4 orbit crosses transversally between the two frequency modes and is asymptotically stable (\cref{teo1} in \cref{sec:idnsps}). For large $a$ the period 4 orbit has a segment of sliding along the discontinuity threshold $y=0$, and nearby orbits collapse onto it in finite time (\cref{teo2} in \cref{sec:idsps}). 

The `nonlinear switching' system, where the discontinuity is governed by \cref{3non}, exhibits more complex dynamics. There are no periodic orbits that switch transversally between frequencies (\cref{inlnsps} in \cref{sec:idnl}), and indeed it becomes impossible for solutions of \cref{1d} to switch frequencies at late times (\cref{nls} in \cref{sec:idnl}). A period 4 orbit emerges (\cref{nlps} in \cref{sec:idnl}) that exists only in the lower frequency state and on the switching threshold, but under regularization is revealed to penetrate less and less deeply into the `switching layer' on each period. We can summarize these by saying that the nonlinear system exhibits:
\begin{itemize}
\item ageing -- the multiplier $\lambda$ interacts in a non-trivial way with the independent variable $t$ in \cref{3non}, such that the dynamics at $y=0$ changes qualitatively, and irreversibly, as $t$ increases. 
\item multiple timescale phenomena -- the nonlinear dependence on $\lambda$ in \cref{3non} creates intricate layering of slow manifolds governing the behaviour at $y=0$. 
\end{itemize} 
In \cref{1d} these have two main effects, meaning that the later a trajectory reaches $y =0$, the longer it may remain there, and also meaning the system asymptotes to but never exactly achieves periodicity. For the full system \cref{3sys} our preliminary work suggests that this is further complicated by relaxation oscillations and mixed mode oscillations, which will be explored in follow-up work.

Multiple timescales arise in realizing that to study solutions of a piecewise-smooth system in general, we must {regularize} the discontinuity in some way. 
In regularization, the {\it switching threshold} is blown up in some way into a {\it switching layer}, by means of which solutions transition through the discontinuity in a well-defined manner. We shall show that the distinction between the linear and nonlinear systems from \cref{3lin} and \cref{3non} remains under regularization, with the periodic orbits of the linear switching system remaining periodic (\cref{teo3} and \cref{teo4} in \cref{sec:rdl}), while in the nonlinear system there persists an aperiodic oscillation lying close to one of the original periodic orbit (\cref{lasteo} and \cref{last} in \cref{sec:rdnl}). 

The most studied method of regularization in recent years has been that of Sotomayor-Teixeira \cite{st96}, which involves smoothing out the discontinuity by replacing $\lambda$ in \cref{3} with a smooth {\it transition} function $\lambda\mapsto\phi(y/\eps)$, where $\phi$ is monotonic and differentiable for small $\eps>0$, such that $\phi(y/\eps)\rightarrow\sign(y)$ as $\eps\rightarrow0$. A more direct approach discussed at length in \cite{j15hidden} is to let $\lambda$ constitute a blow up of the variable $y$ itself, by letting $y=\eps\lambda$ for small $\eps>0$. The dynamics of the multiplier $\lambda$ can be studied on the well-defined layer interval $[-1,+1]$, while $y$ collapses to $y\in\eps[-1,+1]\rightarrow0$ as $\eps\rightarrow0$. In either situation we obtain a differentiable problem that is a singular perturbation of the discontinuous system for small $\eps>0$ (and the two methods yield topologically equivalent systems, see \cite{j18book}). To emphasize that `standard' regularization methods do not resolve the contradiction between the linear and nonlinear switching systems we shall use the more common Sotomayor-Teixeira method.

Our approach will therefore be as follows. In \cref{sec:id} we study the dynamics of the linear system obtained using \cref{3lin}, and then the nonlinear system obtained using \cref{3non}. Then in \cref{sec:rd} we study the regularization of these two systems. Some concluding remarks are made in \cref{sec:conc}.

\section{Preliminaries: dynamics of the piecewise-smooth system} \label{sec:id}

We begin by re-writing \cref{1d} as an autonomous piecewise-smooth system, 
\begin{subequations} \label{s1d}
\begin{align}
\label{s1da} \dot x&=1,
\\
\label{s1db} \dot y&=-ay-f_i(x,\lambda),
\end{align}
\end{subequations}
where $a>0$. The forcing $f_i$ takes one of the two forms from \cref{3}, in which we now set $\omega=\omega_+=3/2$ for $y>0$ and $\omega=\omega_-=1/2$ for $y<0$. The functions $f_L(x,\lambda)$ given by \cref{3lin}, and $f_R(x,\lambda)$ given by \cref{3non}, can then be more concisely written as
\begin{subequations}\label{1f}
\begin{align}
f_L(x,\lambda)&=\left[1+\left(1+\lambda\right)\cos \pi x \right]\sin \sfrac{\pi x}{2}
\;,\label{1flin}\\
f_N(x,\lambda)&=\sin\bb{\pi x(1+\hf\lambda)}
\;,\qquad\;\;\label{1fnon}
\end{align}
\end{subequations}
where $\lambda=\sign(y)$, and where $i=L$ or $N$ indices the linear or nonlinear models of switching, respectively. For $y\neq0$ these both give $f_i(x,\lambda)=\sin(\pi\omega x)$, with frequency $\omega=\omega_+$ for $y>0$ and $\omega=\omega_-$ for $y<0$. We will mainly work in terms of the constants $\omega_\pm$ rather than their numerical values for convenience. Again we specify the $\sign$ function as taking values $+1$ for $ y>0$, $-1$ for $y<0$, and $\lambda\in(-1,+1)$ for $ y=0$; the value on $y=0$ will be considered more closely later.

Let us first sketch out the key regions and features separating different modes of dynamics. We do this in \cref{sec:S}, then give expressions for solutions in the two frequency modes $\omega=\omega_\pm$ in \cref{sec:prelim}, before seeking periodic orbits of the two alternative systems described by \cref{1flin} and \cref{1fnon} in \cref{sec:idl} and \cref{sec:idnl}, respectively. 

\subsection{Regions and key features}\label{sec:S}

Let  $S_\pm$ denote the upper and lower half-planes of $(x,y)\in\mathbb{R}^2$, i.e.
\begin{equation} \label{ulhp} 
S_\pm=\left\{(x,y) \in \mathbb{R}^2\;:\; \pm y > 0 \right\}\;,\;
\end{equation}
with the boundary or {\it switching threshold} between them denoted
\begin{equation}
S_0=\left\{(x,y) \in \mathbb{R}^2\;:\; y= 0 \right\} \;.\quad
\end{equation}
We denote the closure of $S_\pm\cup S_0$ as $\overline S_\pm$. 
The switching threshold $S_0$ itself can be divided into regions where the vector fields in $S_\pm$ point towards or away from $S_0$. Since the subsystems on $S_+$ and $S_-$ are periodic, repeating every $\Delta x={4}/{3}$ and $\Delta x=4$ respectively, it is enough to define these regions on the interval $0\le x\le 4$. We define
\begin{subequations}
\begin{align}
S_{0A}=&\cc{(x,y)\in S_0\;:\;x\in(\sfrac{8}{3},\sfrac{10}{3})}\;,\\
S_{0R}=&\cc{(x,y)\in S_0\;:\;x\in(\sfrac{2}{3},\sfrac{4}{3})}\;,\\
S_{0C}=&\cc{(x,y)\in S_0\;:\;x\in(0,\sfrac{2}{3}) \cup (\sfrac{4}{3},2) \cup (2,\sfrac{8}{3}) \cup (\sfrac{10}{3},4)}\;.
\end{align}
\end{subequations}
As illustrated in \cref{fig:sign}, these portion the switching threshold into:
\begin{itemize}
\item {\it attracting regions} $S_{0A}$, where the fields in $S_\pm$ point towards $S_0$, 
\item {\it repelling regions} $S_{0R}$, where the fields in $S_\pm$ point away from $S_0$, and
\item {\it crossing regions} $S_{0C}$ everywhere else. 
\end{itemize} 
The boundaries between these are places where the vector fields in $S_\pm$ are tangent to $S_0$, i.e. where $\dot y=y=0$. 
This occurs on two sets (associated with the two frequency values $\omega=\omega_\pm$), given by
\begin{subequations}
\begin{align}
T_+&=\cc{(x,y)\in S_0\;:\;x=2n/3}\;,\\ T_-&=\cc{(x,y)\in S_0\;:\;x=2n}\;,
\end{align}
\end{subequations}
such that $\dot y=\sin(\pi\omega_+x)=0$ on $T_+$ and $\dot y=\sin(\pi\omega_-x)=0$ on $T_-$. 


\begin{figure}[t]\centering
\includegraphics[width=0.7\textwidth]{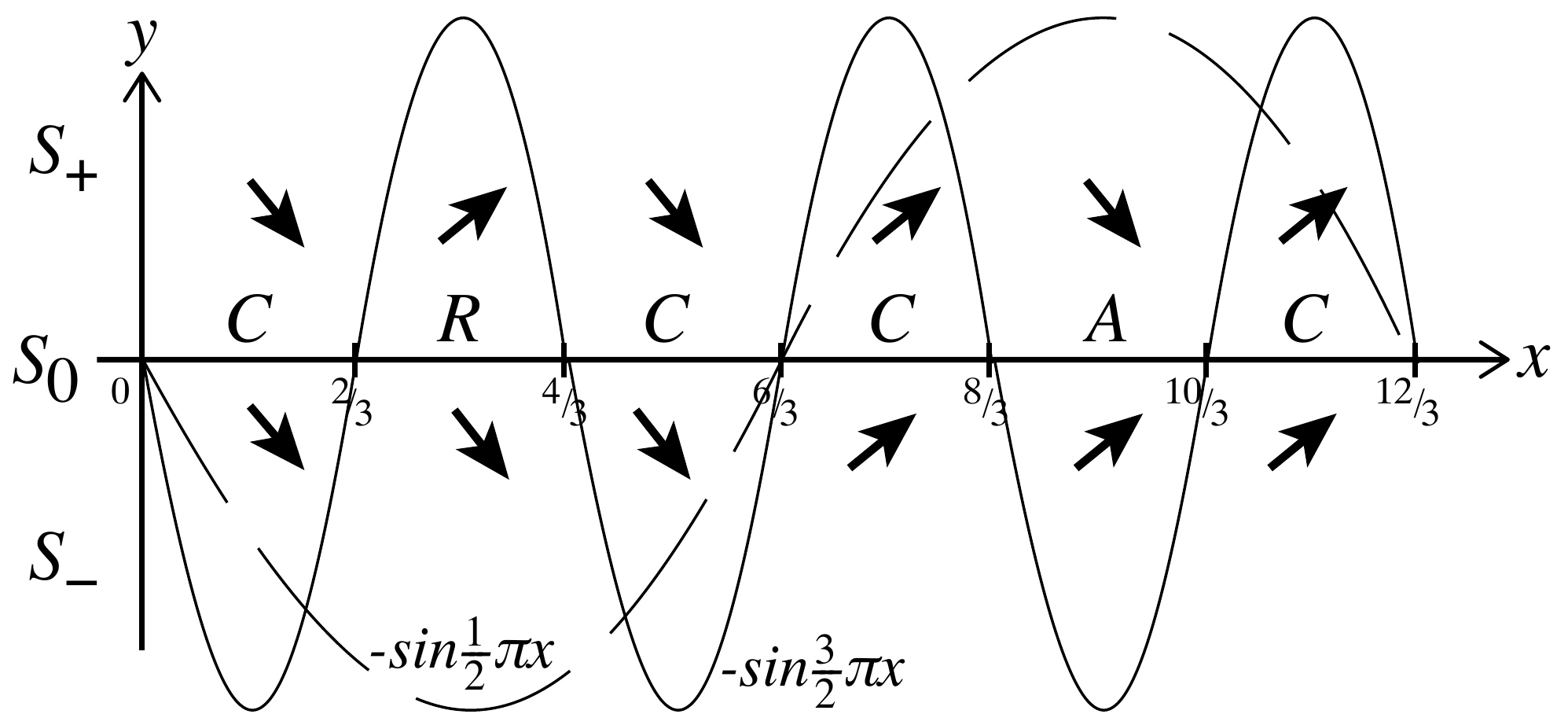}
\caption{The half-planes $S_\pm$ and the switching threshold $S_0$, including the regions of attracting ($A$) and repelling ($R$) sliding, and crossing ($C$). The function $-\sin (\omega\pi x)$ for $\omega=\hf,\sfrac32$, are also shown, along with the directions of the vector fields in $S_\pm$. }\label{fig:sign}
\end{figure} 


In the region $S_{0C}$ we can see that a solution may {\it cross} $S_0$ by entering and exiting $S_0$ from a single point. Alternatively a solution may {\it slide} along any part of $S_0$ where it can satisfy $\dot y=0$ on $y=0$, hence having dynamics
\begin{align}\label{inv}
\dot x&=1\;,\qquad y(t)=0\;,
\end{align}
We define a {\it sliding manifold} as a subset of $S_0$ satisfying
\begin{equation} \label{smi}
\Lambda^{i}:=\left\{\;(x,y) \in S_0,\; \lambda=s(x)\in(-1,+1)\;:\; f_i\left(x,s(x)\right)=0\   \right\}\;,
\end{equation}
such that at any point on $\Lambda^{i}$ there exist solutions to \cref{inv}. The superscript `$i$' indicates that the manifold is associated with either the linear or nonlinear combinations in \cref{1f}. 

In the linear switching system the sliding manifolds $\Lambda^{L}$ simply consist of the regions $S_{0A}$ and $S_{0R}$ (this is a standard result of such {\it Filippov systems}, see e.g. \cite{f88,krg03}, and can be seen by a simple application of the intermediate value theorem as we show at the start of \cref{sec:idl}). In the nonlinear system, $\Lambda^{N}$ may have multiple solution branches not only on $S_{0A}$ and $S_{0R}$, but also on $S_{0C}$.

Periodic solutions will play an important role in the oscillator so we shall make the following distinction. 

\begin{definition} \label{s_ns_ps} A periodic solution of \cref{s1d} is a {\emph {sliding periodic orbit}} if part of it lies on the sliding manifold $\Lambda^{i}$, otherwise it is a {\emph {non-sliding periodic orbit}}.
\end{definition}

These behaviours are illustrated in \cref{fig:a=2} for the linear switching system. 

\begin{figure}[t]\centering
\includegraphics[width=0.8\textwidth]{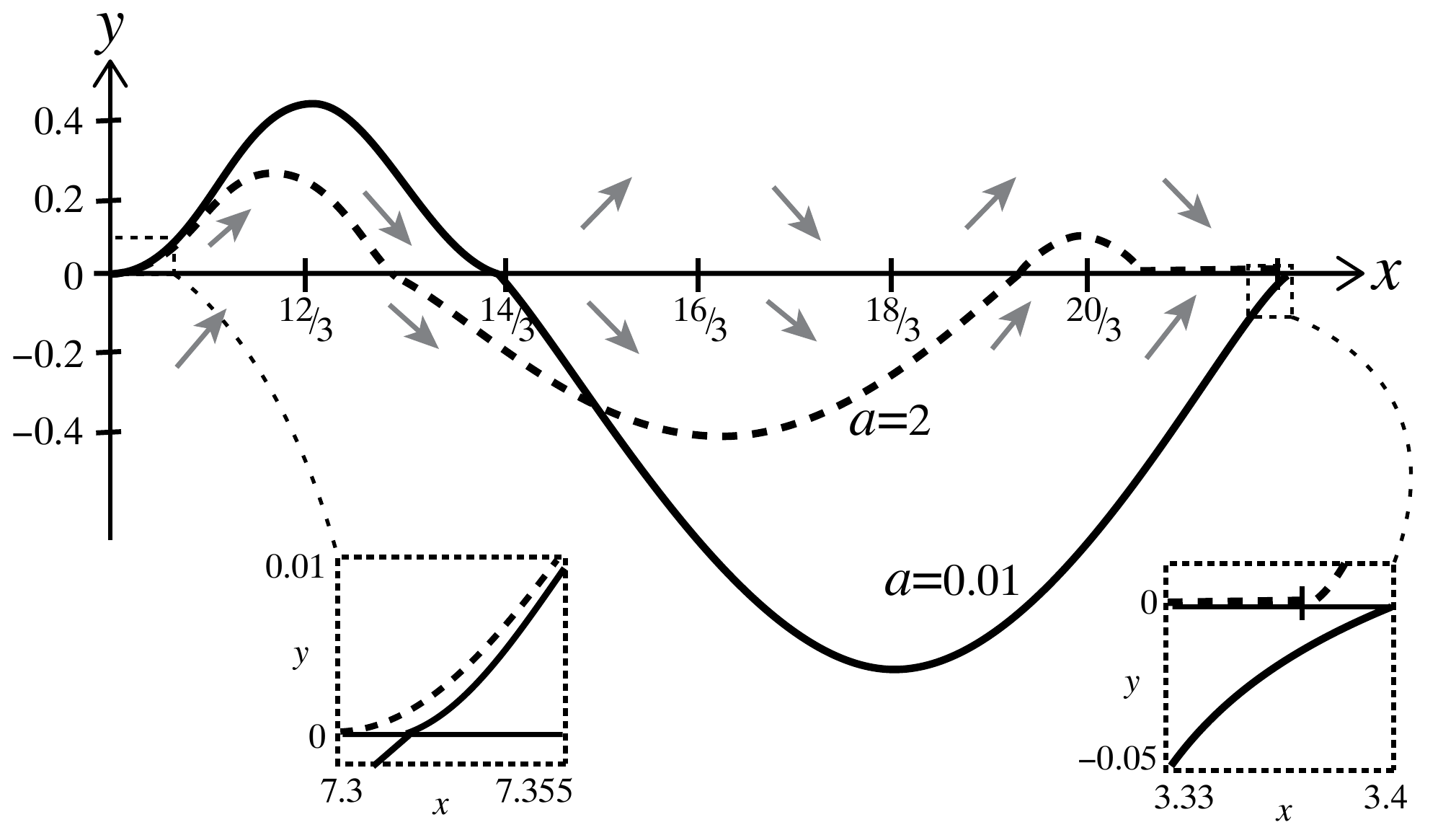}
\caption{Examples of periodic orbits for the linear switching system. A sliding $4$-periodic solution (dotted) is shown passing through point $\left(x,y\right)=\left(\sfrac{10}{3},0\right)$ which lies on the right hand side extremum of the first region $S_{0A}$ in $x>0$, and returning to the next region $S_{0A}$ after two crossings. A non-sliding $4$-periodic solution (full curve) is also shown starting at $\left(x,y\right)=\left(3.355
,0\right)$ and returns to $\left(x,y\right)=\left(7.35
,0\right)$ after one crossing (both in regions $S_{0C}$). Magnifications of the solutions at the start and end of this interval are shown. Arrows indicate the directions of the vector fields in $S_\pm$. (Values given to 4 sig. figs.)}\label{fig:a=2}
\end{figure} 

Explicit expressions can be found for the partial solutions of \cref{s1d} in the two separate half planes $S_\pm$, and we begin by finding these in \cref{sec:prelim}. These will be the same whether we use \cref{1flin} or \cref{1fnon} (and will still apply in \cref{sec:rd} to the regularized system), and we use them to find periodic orbits in \cref{sec:idl} and \cref{sec:idnl}. 
We must then add to this the dynamics on $S_0$, which we do for the linear and nonlinear switching systems respectively in \cref{sec:idl} and \cref{sec:idnl}. 

\subsection{Orbits in the two frequency modes, $S_\pm$} \label{sec:prelim}

Let $y=Y_\pm\left(x,x_i\right)$ denote a solution of \cref{s1d} that starts from a point $(x_i,0)\in S_0$, so that $Y_\pm\left(x_i,x_i\right)=0$, and the `$\pm$' denotes whether the orbit departs $S_0$ into $S_+$ or $S_-$ (orbits remaining on $S_0$ will be considered in \crefs{sec:idl}{sec:idnl}). 

The evolution in either half plane $S_\pm$ is governed simply by \cref{s1d} with $\lambda=\pm1$ in $S_\pm$. Solving in each subsystem separately gives
\begin{equation} \label{sol1}
\begin{aligned}
 Y_\pm\left(x,x_i\right)=& \frac{1}{\omega_\pm^2\pi^2+a^2}\Big[
 \;\;\omega_\pm\pi\cos(\omega_\pm\pi x)-a\sin(\omega_\pm\pi x)\;\;+\\
& \qquad 
e^{-a(x-x_i)}\cc{a\sin(\omega_\pm\pi x_i)-\omega_\pm\pi\cos(\omega_\pm\pi x_i)}\;\Big], 
\end{aligned}
\end{equation}
with $x \in \left[x_i,x_{i+1}\right]$, where $x_{i+1}$ is the next point instersection point with $S_0$, i.e. at which $Y_\pm(x_{i+1},x_i)=0$ and $x_{i+1}>x_i$. 

Setting $\bar{x}=x-x_i$ in \cref{sol1} and grouping trigonometric terms yields $Y_\pm(x_i+\bar x,x_i)=y_\pm(\bar x,x_i)$, where
\begin{equation} \label{sol3} 
y_\pm\left(\bar{x},x_i\right)=\frac{e^{-a\bar{x}}\sin \varphi_{\pm}\left(x_i\right)- \sin \left(\omega_\pm\pi\bar x+\varphi_{\pm}\left(x_i\right)\right)}{\sqrt{\omega_\pm^2\pi^2+a^2}}\;,
\end{equation}
in terms of a function
\begin{equation}\label{varphi} 
\varphi_{\pm}\left(x_i\right)=\omega_\pm\pi x_i-\phi_\pm\;, \qquad\qquad\quad
\end{equation}
and constants $\phi_\pm$ defined by
\begin{equation} \label{fik}
\tan \phi_\pm=\omega_\pm\pi/a\;, \qquad \phi_\pm \in (0,\sfrac{\pi}{2})\;.
\end{equation}

More useful than the solutions themselves is the map through $S_+$ or $S_-$ between successive contact points with $S_0$. 
Let $P_\pm^a : \ \mathbb{R} \rightarrow \mathbb{R}$
be defined as
\begin{equation} \label{ppm} 
P_\pm^a \left(x_i\right)=x_i+ \bar{x}_{i+1}\;,
\end{equation}
(then  $x_{i+1}$ satisfies $x_{i+1}=P_\pm^a \left(x_i\right)$). 
Since $a>0$, any trajectory will eventually hit $S_0$ in finite time, so there always exists a next intersection point $x_{i+1}$ following any $x_i$, and for a certain $y_\pm\left(\bar{x},x_i\right)$ it is given by
\begin{align} \label{bart}  
\bar{x}_{i+1} &= \min_{\bar{x}}\left\{\bar{x} \in \mathbb{R}^+; \ y_\pm\left(\bar{x},x_i\right)=0\right\}\nonumber\\
&=\min_{\bar{x}}\left\{\bar{x} \in \mathbb{R}^+; \ h_{\pm}\left(\bar{x},x_i\right)=0\right\}\;,
\end{align}
in terms of a function
\begin{equation} \label{h}
h_{\pm}\left(\bar{x},x_i\right)=e^{-a\bar{x}}\sin \varphi_{\pm}\left(x_i\right)- \sin \left(\omega_\pm\pi\bar x+\varphi_{\pm}\left(x_i\right)\right)\;.
\end{equation}

Although $h_{\pm}(\bar{x},x_i)=0$ is a transcendental equation and explicit solutions cannot be computed, the zeros of these functions will be useful to bound such contact points.
Let
\begin{align} \label{ha} h_{\pm}^0\left(\bar{x},x_i\right):=&\sin \varphi_{\pm}\left(x_i\right)- \sin \left(\omega_\pm\pi\bar x+\varphi_{\pm}\left(x_i\right)\right), \\
\label{hb} h_{\pm}^{\infty}\left(\bar{x},x_i\right):=& - \sin \left(\omega_\pm\pi\bar x+\varphi_{\pm}\left(x_i\right)\right).
\end{align}
Solving for $\bar x$, the zeroes of $h_{\pm}^0\left(\bar{x},x_i\right)$ lie at
\begin{align} \label{tha}
 \bar{x}= \frac{2n}{\omega_\pm} \quad{\rm or}\quad \frac{2n+1}{\omega_\pm}+\frac{2\phi_{\pm}}{\omega_\pm\pi}-2x_i\;, \qquad  n=0,1,2\ldots,
\end{align}
and the zeroes of $h_{\pm}^\infty\left(\bar{x},x_i\right)$ lie at
\begin{align} \label{thb} 
\bar{x}=\frac{n}{\omega_\pm}-\frac{\varphi_{\pm}\left(x_i\right)}{\omega_\pm\pi}=\frac{n}{\omega_\pm}+\frac{\phi_{\pm}}{\omega_\pm\pi}-x_i, \ \ n=0,1,2\ldots
\end{align}

Although we are interested only in the case $a>0$, we note that for $a=0$ the equation $h_{\pm}\left(\bar{x},x_i\right)=0$ is solvable, and implies $\phi_\pm=\sfrac{\pi}{2}$, so the map reduces to
\begin{equation} \label{P0}
P_\pm^0\left(x_i\right)=\sfrac{2}{\omega_\pm}\left(1+\left \lfloor{\omega_\pm x_i}\right \rfloor\right)-x_i
\end{equation}
where $\lfloor u\rfloor$ denotes the largest integer such that $\lfloor u\rfloor\le u$. 

\section{Linear switching} \label{sec:idl}

Let us assume that the forcing is defined by \cref{1flin}. The system \cref{s1d} can then be re-written as
\begin{subequations} \label{invlin}
\begin{align}
\label{invlina} \dot x&=1,
\\
\label{invlinb} \dot y&=-ay-\left[1+\left(1+\lambda\right)\cos \pi x \right]\sin \sfrac{\pi x}{2},
\end{align}
\end{subequations}
with $\lambda = \sign(y)$ in $S_\pm$, and $\lambda \in \left(-1,1\right)$ in $S_0$.

\begin{proposition} The sliding manifolds of \cref{invlin} are given by
\begin{equation} \label{slin} \Lambda^{L}=\left\{\left(x,0\right) \in \mathbb{R}^2\;:\; x \in \ \left(\sfrac{2}{3}+2n,\sfrac{4}{3}+2n\right), \ n \in \mathbb{N} \right\}.
\end{equation}
\end{proposition}

\begin{proof} It follows from \cref{smi} and \cref{invlinb} that sliding occurs in intervals on $S_0$ where $x$ can satisfy
\begin{align}
  \left[1+\left(1+\lambda\right)\cos \pi x\right]\sin\sfrac{\pi x}{2}=0\quad{\rm for}\quad \lambda \in (-1,1)\;.
\end{align}
There are isolated solution points when $\sin\sfrac{\pi x}2=0$, i.e. at $x=2n$ for $n\in\mathbb N$, but these do not give motion along $S_0$. Assuming $\sin\sfrac{\pi x}2\neq0$, then $x$ must satisfy
\begin{align}
 -1 < \lambda=-1-{\sec \pi x} < 1
\end{align}
hence $0 < -{\sec \pi x} < 2$, implying $\cos \pi x < -\sfrac{1}{2}$, so
\begin{align}
x \in \left(\sfrac{2}{3}+2n,\sfrac{4}{3}+2n\right), \ \ n \in \mathbb{N}\;.
\end{align}
\end{proof}
Notice that the sets $\Lambda^{L}$ coincide with the regions $S_{0A}$ and $S_{0R}$ obtained in \cref{sec:id} and illustrated in \cref{fig:sign}.



Let us now seek periodic orbits that exist with or without sliding in this system, making use of these sliding manifolds on $S_0$, and the function $P$ on $S_\pm$ from \cref{ppm}. 


\subsection{Non-sliding periodic solutions} \label{sec:idnsps}


Using the functions $P_\pm^a$ from \cref{ppm}, let us define
\begin{equation} \label{Pdisc}
P(x,a):=\left(P_+^a \circ P_-^a\right)(x)\;,
\end{equation}
and also define intervals $I_-:=\left(0,\sfrac{2}{3}\right)$ and $I_+:=\left(\sfrac{10}{3},4\right)$.

\begin{lemma} \label{lem1} The mapping  $P(x,0)$ is $4$-periodic for every $x \in I_-$.
\end{lemma}

\begin{proof} As can be seen from \cref{fig:sign}, solutions of \cref{invlin} with initial condition $y(x_i)=0$, $x_i \in I_-$, can evolve through the half-plane $S_-$. Using \cref{P0} it follows that
\begin{equation} \label{p-} P_-^0(x)=4\left(1+\left \lfloor{\sfrac{x}{2}}\right \rfloor\right)-x=4-x \in I_+\;, \quad \forall x \in I_-\;.\end{equation}
Similarly \cref{fig:sign} indicates that solutions with initial condition $y(x_i)=0$, $x_i \in I_+$, can evolve through the half-plane $S_+$. Using \cref{P0} again it follows that
\begin{align} \nonumber P(x,0)&=P_+^0(4-x,0)=\sfrac{4}{3}\left(1+\left \lfloor{\sfrac{3(4-x)}{2}}\right \rfloor\right)-(4-x)\\ \label{p+} &=\sfrac{4}{3}\left(1+5\right)-(4-x)=4+x\;.
\end{align} 
\end{proof}

\begin{lemma} \label{lem2} There exists $a_l \in \mathbb{R}^+$, $a_l \ll 1$, such that the roots of
\begin{align}
\Delta (x,a):=P(x,a)-(x+4)=0
\end{align} define a function $x=x(a)$ for all $a \in \left(0,a_l\right)$, which satisfies $x(0)=x_0$, with $x_0$ in $I_-$ being a solution of
\begin{equation} \label{partialP}  \sfrac{32}{9\pi}+\sfrac{2x_0}{3}{\cot\frac{3\pi x_0}{2}}+\left(4-2x_0\right){\cot\frac{\pi x_0}{2}}=0\;.\end{equation}
\end{lemma}

\begin{proof} The mean value theorem implies that
\begin{equation} \label{del}  \Delta (x,a)-\Delta (x,0)=(a-0)\frac{\partial \Delta}{\partial a}\left(x,\xi(a)\right)\;, \quad 0<\xi(a)<a\;.
\end{equation}
However, as $\Delta (x,0)=P(x,0)-(x+4)=0$ by \cref{lem1}, the equation \cref{del} implies
\begin{equation} \label{del2} \Delta (x,a)=a\frac{\partial P}{\partial a}\left(x,\xi(a)\right), \quad 0<\xi(a)<a.\end{equation}
Notice also that $\xi(0)=0$, so again by the mean value theorem,
\begin{equation} \label{inmind}  \frac{\partial P}{\partial a}\left(x,\xi(a)\right)=\frac{\partial P}{\partial a}\left(x,0\right)+a\frac{\partial^2 P}{\partial a^2}\left(x,\bar \xi(a)\right), \quad 0<\bar \xi(a)<a.\end{equation}
Now, by the implicit function theorem, if there exists $x_0 \in I_-$ such that
\begin{align}\label{c12}
 \frac{\partial P}{\partial a}\left(x_0,0\right)\;=\;0\;\neq\;
\frac{\partial^2 P}{\partial x \partial a}\left(x_0,0\right)\;,
\end{align}
then there exists $a_l \in \mathbb{R}^+$, $a_l \ll 1$, such that a function $x=x(a)$ is defined in $(0,a_l)$ satisfying $x(0)=x_0$ and with $\frac{\partial P}{\partial a}\left(x(a),\xi(a)\right)=0$. Returning to \cref{del2}, the condition $ \Delta (x,a)=P(x,a)-(x+4)=0$ is thus equivalent to
\begin{align}
 \frac{\partial P}{\partial a}\left(x,\xi(a)\right)=0\quad{\rm for}\quad a > 0\;.
\end{align}
Taking into account \cref{h} and \cref{varphi}, $\frac{\partial P}{\partial a}\left(x_0,0\right)$ in \cref{c12} is to be obtained from
\begin{align} \label{h-} & h_{-}\left(P_-^a(x)-x,x\right)=0,\\
              \label{h+} & h_{+}\left(P(x,a)-P_-^a(x),P_-^a(x)\right)=0.
\end{align}
Using implicit derivation, with the relations \cref{p-} and \cref{p+}, and substituting $\omega_+=3/2$ and $\omega_-=1/2$ to simplify, some lengthy but straightforward algebra yields
\begin{align}
 \sfrac{\partial P}{\partial a}\left(x_0,0\right)=\sfrac{2}{\pi}\left(\sfrac{32}{9\pi}+\sfrac{2x_0}{3}{\cot\sfrac{3\pi x_0}{2}}+\left(4-2x_0\right){\cot\sfrac{\pi x_0}{2}}\right)\;,
\end{align}
which exists and is continuous for all $x_0 \in I_-$, and $\frac{\partial P}{\partial a}\left(x_0,0\right)=0$, $x_0 \in I_-$, if and only if \cref{partialP} is fulfilled. Moreover,
\begin{align}
\frac{2}{\pi}\frac{\partial P}{\partial a}\left(\sfrac{1}{2},0\right)=\frac{8\left(4+3\pi\right)}{9\pi}>0\;,
\end{align}
and 
\begin{align}
\lim_{x_0 \rightarrow \sfrac{2}{3}^-}\frac{2}{\pi}\frac{\partial P}{\partial a}\left(x_0,0\right)=-\infty< 0\;,
\end{align}
and Bolzano's theorem guarantees that \cref{partialP} has a solution in $\left(\sfrac{1}{2},\sfrac{2}{3}\right) \subset I_-$. 
Finally, for all $x_0 \in I_-$,
\begin{align} \sfrac{\partial^2 P}{\partial x \partial a}\left(x_0,0\right)= &
\sfrac{4}{\pi}\bb{\sfrac13\cot\sfrac{3\pi x_0}{2}
-\cot\sfrac{\pi x_0}{2}}\\&
-2x_0\cosec^2\sfrac{3\pi x_0}{2}
-(4-2x_0)\cosec^2\sfrac{\pi x_0}{2}<0\;,\end{align}
and the result follows. 
\end{proof}

\begin{lemma} \label{lem3} For all $a \in \left(0,a_l\right) \subset \mathbb{R}^+$, there exists $\mu(a) \in \mathbb{R}^+$, $0<\mu(a)\ll 1$, such that
\begin{equation} \label{partialPx} 0 <   \frac{\partial P}{\partial x} (x(a),a)  < 1, \quad \forall x \in \left(x(a)-\mu(a),x(a)+\mu(a)\right),  \end{equation}
where $x=x(a)$ denotes the function defined by $P(x,a)-(x+4)=0$.
\end{lemma}

\begin{proof} Taking implicit derivation with respect to $x$ in \cref{h-}-\cref{h+} gives
\begin{align} \frac{\partial P_-^a(x)}{\partial x}&=\frac{\sin \sfrac{\pi x}{2}}{\sin \frac{\pi P_-^a(x)}{2}}e^{-a\left(P_-^a(x)-x\right)}\;, \\
\frac{\partial P(x,a)}{\partial x}&=\frac{\sin \frac{3\pi P_-^a(x)}{2}}{\sin \frac{3\pi P(x,a)}{2}}\cdot\frac{\partial P_-^a(x)}{\partial x}\cdot e^{-a\left(P(x,a)-P_-^a(x)\right)}\;.
\end{align}
Combining both expressions for $x=x(a)$, noting $P(x(a),a)=4+x(a)$, and using the sine triple-angle identity, it follows that
\begin{equation} \label{partialPxx} \frac{\partial P}{\partial x}(x(a),a)=\frac{3-4\sin^2 \frac{\pi P_-^a\left((x(a)\right)}{2}}{3-4\sin^2 \sfrac{\pi x(a)}{2}}e^{-4a}. \end{equation}
Moreover, from \cref{p-} we have that $P^0_-(x)=4-x$, hence by continuity there exists $\mu(a) \in \mathbb{R}$, $|\mu(a)|\ll 1$, such that $P^a_-(x)=4-x(a)+\mu(a)$, and \cref{partialPxx} becomes
\begin{equation}\label{onemore} \frac{\partial P}{\partial x}(x(a),a)=\frac{3-4\sin^2 \sfrac{\pi x(a)}{2}+\pi \mu(a) \sin \pi x(a)}{3-4\sin^2 \sfrac{\pi x(a)}{2}}e^{-4a}+O\left(\mu^2(a)\right).\end{equation}
Substituting $x=x(a)$ and $P^a_-(x)=4-x(a)+ \mu(a)$ into \cref{h-} results in
\begin{align} 
&e^{-a\left(4-2x(a)+ \mu(a)\right)}\left(2a\sin\sfrac{\pi x(a)}{2}-\pi\cos\sfrac{\pi x(a)}{2}\right)\nonumber\\
&=2a\sin\sfrac{\pi \left(\mu(a)-x(a)\right)}{2}-\pi\cos\sfrac{\pi \left(\mu(a)-x(a)\right)}{2}\;,
\end{align}
which can also be written as
\begin{align} 
&e^{-a\left(4-2x(a)+\mu(a)\right)}\left(2a\tan\sfrac{\pi x(a)}{2}-\pi\right)\nonumber\\
&=\left(2a-\pi \tan\sfrac{\pi x(a)}{2}\right)\tan\sfrac{\pi \mu(a)}{2}-\left(2a\tan\sfrac{\pi x(a)}{2}+\pi \right).
\end{align}
Then expanding $e^{-a\left(4-2x+\mu(a)\right)}$ and $\tan\sfrac{\pi \mu(a)}{2}$ up to first order yields
\begin{align}
  \mu(a)=-4a\frac{2+\left(2-x(a)\right)\left(\pi  \cot \sfrac{\pi x(a)}{2} -2a\right)}{\pi^2-4a^2}\;,
\end{align}
and using \cref{onemore} we obtain
\begin{align}
 \mu(a) \sin \pi x(a)=-8a\sfrac{\sin \pi x(a)+\left(2-x(a)\right)\left(\pi  \cos^2 \sfrac{\pi x(a)}{2} -a\sin \pi x(a)\right)}{\pi^2-4a^2}\;.
\end{align}
Since $x(a) \in I_-=\left(0,\sfrac{2}{3}\right)$ this therefore implies that
\begin{align}
 \mu(a) \sin \pi x(a) \rightarrow 0^- \quad \mbox{when} \ \ a \rightarrow 0^+\;,
 \end{align}
and \cref{partialPx} follows directly. 
\end{proof}


An immediate consequence of \cref{lem2} and \cref{lem3} is then the following.

\begin{theorem} \label{teo1} There exists $a_l \in \mathbb{R}^+$, $a_l \ll 1$, such that for all $a \in (0,a_l)$ system \cref{invlin} has a non-sliding periodic solution, which is $4$-periodic and locally asymptotically stable. \qed
\end{theorem}

Numerical computations show that a solution of \cref{partialP} in $\left(0,\sfrac{2}{3}\right)$ is given by $x_0=0.6357545163..$, and that, for example, the non-sliding periodic solution for $a=0.01$ crosses $S_0$ at $x=0.6261249968..$ (to 10 significant figures, simulated in \cref{fig:a=2}).


\subsection{Sliding periodic solutions} \label{sec:idsps}

We can prove the existence of sliding periodic solutions for large values of $a$ by showing that a solution exiting the sliding interval $x\in\left(\frac{8}{3},\sfrac{10}{3}\right)$ through its right hand side extremum, arrives in the next sliding interval, $x\in\left(\frac{20}{3},\frac{22}{3}\right)$ after several crossings of $S_0$.

According to \cref{fig:sign}, a solution starting at $\left(x_0,y_0\right)=\left(\sfrac{10}{3},0\right)$ evolves on $S_+$. Using the function $P_+^a$ defined in \cref{ppm}, let
\begin{align}
 P_+^a\left(x_0\right)=P_+^a\left(\sfrac{10}{3}\right)
\end{align}
be the next crossing point of $S_0$.

\begin{lemma} \label{lems1} For all $a \in \mathbb{R}^+$ we have that
\begin{align}
 P_+^a\left(\sfrac{10}{3}\right) \in \left(4,\sfrac{14}{3}\right)\;.
\end{align}
Moreover there exist $a_1,\delta_1 \in \mathbb{R}^+$, $a_1\gg 1$, $0<\delta_1\ll 1$, such that for all $a \in \left(a_1,+\infty\right)$, we have $P_+^a\left(\sfrac{10}{3}\right) \in \left(4,\delta_1\right)$.
\end{lemma}

\begin{proof} Following \cref{ppm} and \cref{ha}, let
\begin{align}
  x_1=P_+^a\left(\sfrac{10}{3}\right)=\sfrac{10}{3}+\bar{x}_1\;,
\end{align}
be the first positive $x$ at which
\begin{equation} \label{h3103} h_{+}\left(\bar{x},\sfrac{10}{3}\right)=e^{-a\bar{x}}\sin \varphi_{+}\left(\sfrac{10}{3}\right)- \sin \left(\sfrac{3\pi \bar{x}}{2}+\varphi_{+}\left(\sfrac{10}{3}\right)\right)=0\;, \end{equation}
where
\begin{equation} \label{h3103_2} \varphi_{+}\left(\sfrac{10}{3}\right)=\sfrac{3\pi \sfrac{10}{3}}{2}-\phi_+=5\pi-\phi_+ \equiv \pi-\phi_+\;.\end{equation}
We know from \cref{fik} that $\phi_+ \in \left(0,\sfrac{\pi}{2}\right)$, for all $a \in \mathbb{R}^+$, so
\begin{align}
 \varphi_{+}\left(\sfrac{10}{3}\right) \equiv \pi-\phi_+ \in \left(\sfrac{\pi}{2},\pi\right) \quad\Rightarrow \quad  \sin \varphi_{+}\left(\sfrac{10}{3}\right)> 0, \ \forall a \in \mathbb{R}^+\;.
\end{align}
Therefore
\begin{align}
  0<h^{\infty}_{+}\left(\bar{x},\sfrac{10}{3}\right) < h_{+}\left(\bar{x},\sfrac{10}{3}\right) < h_{+}^0\left(\bar{x},\sfrac{10}{3}\right)
\end{align}
in appropriate intervals, yielding $\bar{x}_\infty < \bar{x}_1 < \bar{x}_0$, where $\bar{x}_0$ and $\bar{x}_\infty$ denote the first positive zeros of $h_{+}^0\left(\bar{x},\sfrac{10}{3}\right)$ and $h_{+}^{\infty}\left(\bar{x},\sfrac{10}{3}\right)$, respectively.

Using \cref{tha} we have that
\begin{align}
 \left. \begin{array}{l} \bar{x}=\sfrac{4}{3} \ \ (n=1), \\
                           \bar{x}= \frac{2(2n+1)}{3}+\frac{4\phi_{+}}{3\pi}-2\sfrac{10}{3} \in \left(\sfrac{2}{3},\sfrac{4}{3}\right) \ \ (n=5) \end{array}
   \right\} \quad\Rightarrow \quad \bar{x}_0=\sfrac{4}{3},
\end{align}
while \cref{thb} yields
\begin{align}
 \bar{x}=\sfrac{2n}{3}+\sfrac{2\phi_{+}}{3\pi}-\sfrac{10}{3} \in \left(\sfrac{2}{3},\sfrac{4}{3}\right) \ \ (n=5) \quad\Rightarrow \quad \bar{x}_\infty=\sfrac{2}{3}\;.
\end{align}
Notice that the worst cases for the lower limit (lowest value for $\bar{x}_\infty$) and the upper limit (highest value for $\bar{x}_0$) have been selected. Then
\begin{align}
 \sfrac{2}{3} < \bar{x}_1 < \sfrac{4}{3} \quad\Rightarrow \quad 4 < x_1 < \sfrac{14}{3} \quad\Rightarrow \quad x_1 \in \left(4,\sfrac{14}{3}\right), \quad \forall a \in \mathbb{R}^+\;.
\end{align}
This also implies that
\begin{equation}                \label{ainf}  a \rightarrow +\infty \quad\Rightarrow \quad \bar{x}_1  \rightarrow \bar{x}_\infty^+=\sfrac{2}{3}^+ \quad\Rightarrow \quad x_1  \rightarrow 4^+.
\end{equation}
Then the existence of $a_1$ and $\delta_1$ are guaranteed by continuity.
\end{proof}
%

As can be seen from \cref{fig:sign}, a solution starting at $\left(x_1,0\right)$, with $x_1=P_+^a\left(\sfrac{10}{3}\right)$ $\in$ $\left(4,\sfrac{14}{3}\right)$, will evolve via $S_-$. Hence, the next crossing point of $S_0$ is
\begin{align}
 P_-^a\left(x_1\right)=\left(P_-^a \circ P_+^a \right) \left(\sfrac{10}{3}\right)\;.
\end{align}

\begin{lemma} \label{lems2} There exist $a_2,\delta_2 \in \mathbb{R}^+$, $a_2\gg 1$, $0<\delta_2\ll 1$, such that for all $a \in \left(a_2,+\infty\right)$, we have that $P_-^a\left(x\right) \in \left(6,6+\delta_2\right)$, for all $x \in \left(4,\delta_1\right)$.
\end{lemma}

\begin{proof} In accordance with \cref{ppm} and \cref{ha}, let
\begin{align}
  x_2=P_-^a\left(x_1\right)=x_1+\bar{x}_2=P_+^a\left(\sfrac{10}{3}\right)+\bar{x}_2\;,
\end{align}
with $\bar{x}_2$ being the first positive zero of
\begin{equation} \label{h1x1} h_{-}\left(\bar{x},x_1\right)=e^{-a\bar{x}}\sin \varphi_{-}\left(x_1\right)- \sin \left(\sfrac{\pi \bar{x}}{2}+\varphi_{-}\left(x_1\right)\right)=0\;, \end{equation}
where
\begin{equation} \label{h1x1_2} \varphi_{-}\left(x_1\right)=\sfrac{\pi x_1}{2}-\phi_-\;.\end{equation}
It follows from \cref{fik} and \cref{ainf} that
\begin{align}
 &a \rightarrow +\infty \quad\Rightarrow \quad \left\{ \begin{array}{l} x_1  \rightarrow 4^+ \\ \phi_- \rightarrow 0^+ \end{array} \right\} \nonumber\\
 &\Rightarrow \quad \varphi_{-}\left(4^+\right) \rightarrow 0 \quad\Rightarrow \quad \sin \varphi_{-}\left(4^+\right) \rightarrow 0\;.
 \end{align}
To bound the zeros of \cref{h1x1} for $a \rightarrow +\infty$ requires knowledge of the sign of $\sin \varphi_{-}\left(4^+\right)$. Notice that, using \cref{h3103_2}, equation \cref{h3103} applied for $\bar{x}_1$ can be equivalently written as
\begin{align}
e^{-a\bar{x}_1}\sin \phi_{+}+ \sin \left(\sfrac{3\pi \bar{x}_1}{2}-\phi_{+}\right)=0\;.
\end{align}
Now, obtaining $\sin \phi_{+}$ and $\cos \phi_{+}$ from \cref{fik}, the preceding equation is equivalent to
\begin{equation} \label{unaeq} 3\pi e^{-a\bar{x}_1}+2a \sin \sfrac{3\pi \bar{x}_1}{2}-3\pi \cos \sfrac{3\pi \bar{x}_1}{2}=0\;.\end{equation}
Letting 
\begin{align}
\bar{x}_1=x_1-\sfrac{10}{3}=4+\epsilon-\sfrac{10}{3}=\sfrac{2}{3}+\epsilon 
\end{align}
in \cref{unaeq} for $\epsilon \rightarrow 0^+$ yields
\begin{align}
  3\pi e^{-a\left(\sfrac{2}{3}+\epsilon\right)}-2a \sin \sfrac{3\pi \epsilon}{2}+3\pi \cos \sfrac{3\pi \epsilon}{2}=0\;,
\end{align}
which can be re-arranged to
\begin{align}
 \tan \sfrac{3\pi \epsilon}{2} =\sfrac{3\pi}{2a} \left[1+ {e^{-a\left(\sfrac{2}{3}+\epsilon\right)}}{\sec \sfrac{3\pi \epsilon}{2}}\right]\;.
 \end{align}
Then taking into account that
\begin{align}
 \left. \begin{array}{c} \epsilon \rightarrow 0^+ \\ a \rightarrow + \infty  \end{array} \right\} \quad\Rightarrow \quad {e^{-a\left(\sfrac{2}{3}+\epsilon\right)}}{\sec \sfrac{3\pi \epsilon}{2}} \rightarrow 0 
\end{align}
which shrinks exponentially fast with increasing $a$, we have
\begin{equation} \label{aprox} \tan \sfrac{3\pi \epsilon}{2} \approx \sfrac{3\pi}{2a} \end{equation}
at any order of $\epsilon$. Furthermore,
\begin{align} \sin \varphi_{-}\left(4^+\right) &= \sin \varphi_{-}\left(4+\epsilon\right)\nonumber\\&=\sin \left(\sfrac{\pi \left(4+\epsilon\right)}{2}-\phi_-\right) = \sin \left(\sfrac{\pi \epsilon}{2}-\phi_-\right)\nonumber\\
 &= \frac{1}{\sqrt{\pi^2+4a^2}}\left( 2a \sin \sfrac{\pi \epsilon}{2}-\pi \cos \sfrac{\pi \epsilon}{2} \right)\nonumber\\&=\frac{2a\cos \frac{\pi \epsilon}{2}}{\sqrt{\pi^2+4a^2}}\left(\tan \sfrac{\pi \epsilon}{2}-\sfrac{\pi}{2a} \right).
\end{align}
Now, \cref{aprox} implies
\begin{align}\label{ladabans}
  \tan \sfrac{\pi \epsilon}{2}-\sfrac{\pi}{2a} & \approx \tan \sfrac{\pi \epsilon}{2} - \sfrac{1}{3} \tan \sfrac{3\pi \epsilon}{2} \nonumber\\&= \tan \sfrac{\pi \epsilon}{2} - \frac{3\tan \frac{\pi \epsilon}{2}-\tan^3 \frac{\pi \epsilon}{2}}{3\left(1-3\tan^2 \frac{\pi \epsilon}{2}\right)}\nonumber\\  &=-\frac{8\tan^3 \frac{\pi \epsilon}{2}}{3\left(1-3\tan^2 \frac{\pi \epsilon}{2}\right)} < 0, \ \ \epsilon \rightarrow 0^+,
  \end{align}
yielding $\sin \varphi_{-}\left(4^+\right) \rightarrow 0^-$. 
Therefore,
\begin{align}
  h_{-}^{0} \left(\bar{x}_0,4^+\right) < h_{-}\left(\bar{x}_2,4^+\right)\;,
\end{align}
with $\bar{x}_2 \rightarrow \bar{x}_0^+$. Using again \cref{tha},
\begin{align}
&\left. \begin{array}{l} \bar{x}=4n=4 \ \ (n=1), \\
                           \bar{x}= \left. 2(2n+1)+\frac{4\phi_{1}}{\pi}-2\cdot 4 \right|_{\phi_- \rightarrow 0}=2 \ \ (n=2) \end{array}
   \right\} \nonumber\\
   &\Rightarrow \quad \bar{x}_0=2 \quad\Rightarrow \quad x_2 \rightarrow 6^+.
\end{align}
The results follows immediately by continuity.
\end{proof}

An equivalent procedure implies that for $a \rightarrow 0^+$ there do not exist sliding periodic solutions, as $\left(P_-^a \circ P_+^a \right) \left(\sfrac{10}{3}\right) \rightarrow \frac{22}{3}^+$, and $x=\frac{22}{3}^+$ belongs to the crossing region $S_{0C}$ beyond $\left(\frac{20}{3},\frac{22}{3}\right)$, which is the next sliding interval after $\left(\frac{8}{3},\sfrac{10}{3}\right)$. This result is consistent with \cref{teo1}.

According to \cref{fig:sign}, a solution starting at $\left(x_2,0\right)$ with
\begin{align}
x_2=\left(P_-^a \circ P_+^a\right)\left(\sfrac{10}{3}\right) \in \left(6,6+\delta_2\right)\;,
\end{align} will evolve via $S_+$. Hence the next crossing point of $S_0$ is
\begin{align}
 P_+^a\left(x_2\right)=\left(P_+^a \circ P_-^a \circ P_+^a \right) \left(\sfrac{10}{3}\right)\;.
\end{align}

\begin{lemma} \label{lems3} There exist $a_3,\delta_3 \in \mathbb{R}^+$, $a_3\gg 1$, $0<\delta_3\ll 1$, such that for all $a \in \left(a_3,+\infty\right)$, we have that $P_+^a\left(x\right) \in \left(\frac{20}{3},\frac{20}{3}+\delta_3\right)$ for all $x \in \left(6,6+\delta_2\right)$.
\end{lemma}

\begin{proof} Recalling \cref{ppm} and \cref{ha}, let
\begin{align}
  x_3=P_+^a\left(x_2\right)=x_2+\bar{x}_3=\left(P_-^a \circ P_+^a \right)\left(\sfrac{10}{3}\right)+\bar{x}_3\;,
\end{align}
with $x_3$ being the first positive zero of
\begin{equation} \label{h1x2} h_{+}\left(\bar{x}_3,x_3\right)=e^{-a\bar{x}_3}\sin \varphi_{+}\left(x_2\right)- \sin \left(\sfrac{3\pi \bar{x}_3}{2}+\varphi_{+}\left(x_2\right)\right)=0\;, \end{equation}
where
\begin{equation} \label{h1x2_3} \varphi_{+}\left(x_2\right)=\sfrac{3\pi x_2}{2}-\phi_+\;.\end{equation}

Again in this case, $a \rightarrow +\infty \Rightarrow \phi_+ \rightarrow 0^+$ which, combined with $x_2 \rightarrow 6^+$, yields $\varphi_{+}\left(6^+\right)=\pi^+-0^+$, thus preventing a straightforward assessment of the sign of $\sin \varphi_{+}\left(6^+\right)$. However, a procedure similar to the one used in the proof of \cref{lems2} yields $\sin \varphi_{+}\left(6^+\right) \rightarrow 0^-$. Let $\bar{x}_2=2+\epsilon$ in \cref{h1x1} with $\epsilon \rightarrow 0^+$, which implies
\begin{align}
 e^{-a\left(2+\epsilon\right)}\sin \varphi_{-}\left(4^+\right)+2a \sin \sfrac{\pi \epsilon}{2}-\pi \cos \sfrac{\pi \epsilon}{2}=0\;,
\end{align}
which re-arranges to
\begin{align}
 \tan \frac{\pi \epsilon}{2} =\frac{\pi}{2a}+ \frac{e^{-a\left(2+\epsilon\right)}}{a\cos \frac{\pi \epsilon}{2}}\left|\sin \varphi_{-}\left(4^+\right)\right|\;.
\end{align}
Then taking into account that
\begin{align}
 \left. \begin{array}{c} \epsilon \rightarrow 0^+ \\ a \rightarrow + \infty \\ \sin \varphi_{-}\left(4^+\right) \rightarrow 0  \end{array} \right\} \quad\Rightarrow \quad \frac{e^{-a\left(2+\epsilon\right)}}{a\cos \frac{\pi \epsilon}{2}}\left|\sin \varphi_{-}\left(4^+\right)\right| \rightarrow 0 
\end{align}
decreases exponentially fast with increasing $a$, we have
\begin{equation} \label{aprox2} \tan \sfrac{\pi \epsilon}{2} \approx \sfrac{\pi}{2a} \end{equation}
at any order of $\epsilon$. Furthermore,
\begin{align} \sin \varphi_{+}\left(6^+\right)  &= \sin \varphi_{+}\left(6+\epsilon\right)\nonumber\\&=\sin \left(\sfrac{3\pi \left(6+\epsilon\right)}{2}-\phi_3\right)= -\sin \left(\sfrac{3\pi \epsilon}{2}-\phi_3\right) \nonumber\\
 &= -\frac{1}{\sqrt{9\pi^2+4a^2}}\left( 2a \sin \sfrac{3\pi \epsilon}{2}-3\pi \cos \sfrac{3\pi \epsilon}{2} \right)\nonumber\\ &=\frac{6a\cos \frac{\pi \epsilon}{2}}{\sqrt{9\pi^2+4a^2}}\left(\sfrac{\pi}{2a}-\sfrac{1}{3}\tan \sfrac{3\pi \epsilon}{2} \right).
\end{align}
Now \cref{aprox2} and \cref{ladabans} imply
\begin{align} \sfrac{\pi}{2a}-\sfrac{1}{3}\tan \sfrac{3\pi \epsilon}{2} \approx \tan \sfrac{\pi \epsilon}{2} - \sfrac{1}{3} \tan \sfrac{3\pi \epsilon}{2}  < 0, \ \ \epsilon \rightarrow 0^+,\end{align}
yielding $\sin \varphi_{+}\left(6^+\right) \rightarrow 0^-$. Therefore,
\begin{align}
  h_{+}^0 \left(\bar{x}_0,6^+\right) < h_{+}\left(\bar{x}_3,6^+\right)\;,
\end{align}
with $x_3 \rightarrow x_0^+$. Using  \cref{tha},
\begin{align}
& \left. \begin{array}{l} \bar{x}=\frac{4n}{3}=\sfrac{4}{3} \ \ (n=1), \\
                           \bar{x}= \left. \frac{2(2n+1)}{3}+\frac{4\phi_{+}}{3\pi}-2\cdot 6 \right|_{\phi_+ \rightarrow 0}=\sfrac{2}{3} \ \ (n=6) \end{array}
   \right\} \quad\Rightarrow \quad \bar{x}_0=\sfrac{2}{3} \nonumber\\
 &  \Rightarrow \quad x_3 \rightarrow \sfrac{20}{3}^+\;,
\end{align}
and the result follows by continuity.
\end{proof}

\begin{theorem} \label{teo2} There exists $a_h \in \mathbb{R}^+$, $a_h \gg 1$, such that for all $a \in (a_h,+\infty)$, system \cref{invlin} has a sliding $4$-periodic solution. Moreover, for every $a \in (a_h,+\infty)$ there exists $\nu (a) \in \mathbb{R}^+$ such that solutions $y\left(x,x_i\right)$ of \cref{invlin}, with $x_i \in \left(\frac{8}{3},\sfrac{10}{3}\right) \cup \left(\sfrac{10}{3},\sfrac{10}{3}+\nu (a)\right)$, converge in finite time to the periodic solution $y\left(x,\sfrac{10}{3}\right)$.
\end{theorem}

\begin{proof} The first part of the statement follows from the successive application of \cref{lems1} and \cref{lems3}, with $a_h=\max\{a_1,a_2,a_3\}$, because $\left(\frac{20}{3},\frac{20}{3}+\delta_3\right) \subset \left(\frac{20}{3},\frac{22}{3}\right)$, which is the next attractive sliding interval after $\left(\frac{8}{3},\sfrac{10}{3}\right)$.

The finite time convergence follows immediately for solutions with $x_i \in \left(\frac{8}{3},\sfrac{10}{3}\right)$, because this is an attractive sliding interval and a solution starting therein is such that $y\left(x,x_i\right)=0$, $\forall x \in \left[x_i,\sfrac{10}{3}\right]$, so in at most a distance of $\sfrac{2}{3}$ in $x$ it reaches the periodic solution $y\left(x,\sfrac{10}{3}\right)=0$.

For $x_i \in \left(\sfrac{10}{3},\sfrac{10}{3}+\nu (a)\right)$, \cref{lems1} and \cref{lems3} and the continuous dependence of solutions on initial conditions then imply that, given $a > a_h$, there exists $\delta \left(\nu (a) \right) \in \mathbb{R}^+$, $0<\delta \left(\nu (a) \right)\ll 1$, such that
\begin{align}
 \left(P^a_+ \circ P^a_- \circ P_+^a\right)\left(x_i\right) \in \left(\sfrac{20}{3},\sfrac{20}{3}+\delta \left(\nu (a) \right) \right) \subset \left(\sfrac{20}{3},\sfrac{22}{3}\right)\;
\end{align}
hence in this case the periodic solution is reached in at most a distance 4 in $x$.
\end{proof}

Numerical simulations appear to show that sliding periodic solutions exist for a wide range of values of $a$, failing to appear only for $a \rightarrow 0$. Some examples are simulated in \cref{fig:ints2}.

\begin{figure}[t]\centering
\includegraphics[width=0.8\textwidth]{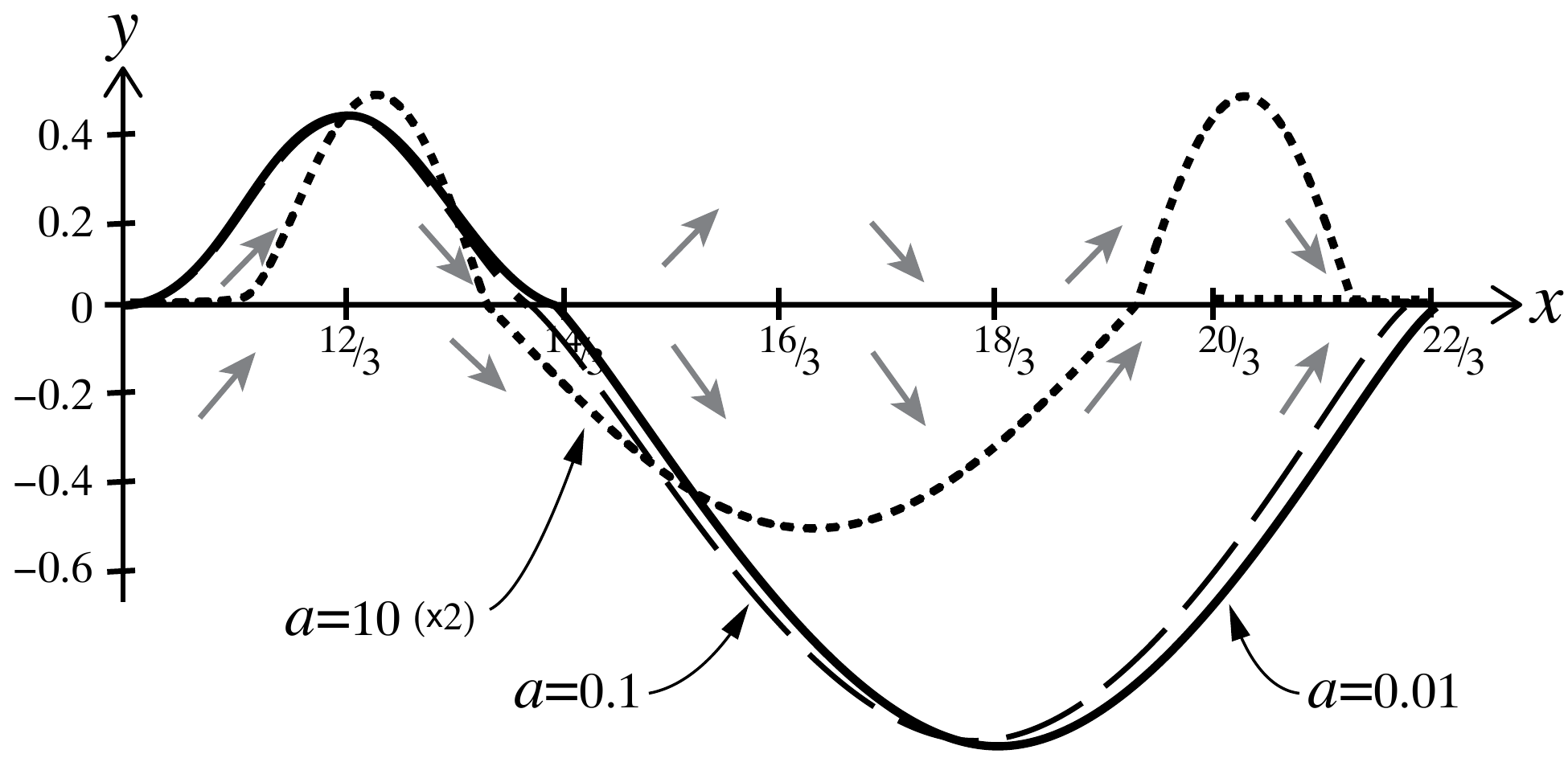}
\caption{Period 4 orbits of the linear switching system, showing: a non-sliding orbit for $a=0.01$ (full curve), an orbit with a small segment of sliding for $a=0.1$ (dashed curve), and a sliding orbit for $a=10$ (dotted curve, vertical scale multiplied by $1/2$ for clarity).}\label{fig:ints2}
\end{figure}


\section{Nonlinear switching} \label{sec:idnl}

Let us assume that the forcing is defined by \cref{1fnon}. Then \cref{s1d} can be written as
\begin{subequations} \label{invnon}
\begin{align}
\label{invnona} \dot x&=1,
\\
\label{invnonb} \dot y&=-ay-\sin \left[\pi x \left(1+\sfrac{1}{2}\lambda\right)\right],
\end{align}
\end{subequations}
with $\lambda = \sign(y)$ in $S_\pm$, and $\lambda \in \left(-1,1\right)$ in $S_0$.

\begin{proposition} The sliding manifolds of \cref{invnon} are given by
\begin{equation} \label{snon} \Lambda^{N}=\left\{\left(x,0\right) \in \mathbb{R}^2\;:\; x \in \left(\sfrac{2n}{3},2n\right), \ n \in \mathbb{N} \setminus\{0\} \right\}.
\end{equation}
\end{proposition}

\begin{proof} It follows from \cref{inv} and \cref{invnonb} that the sliding time intervals are defined by
\begin{align}
 \sin \left(\pi x \left[1+\sfrac{1}{2}\lambda\right]\right)=0, \quad \lambda \in (-1,1)<1\;.
\end{align}
Therefore $-1 < \lambda=2\left(\sfrac{n}{x}-1\right) < 1$, hence
\begin{align}
 x \in \left(\sfrac{2n}{3},2n\right)\quad{\rm for}\quad n \in \mathbb{N}\setminus\{0\}\;.
\end{align}
\end{proof}

\begin{figure}[t]\centering
\includegraphics[width=0.85\textwidth]{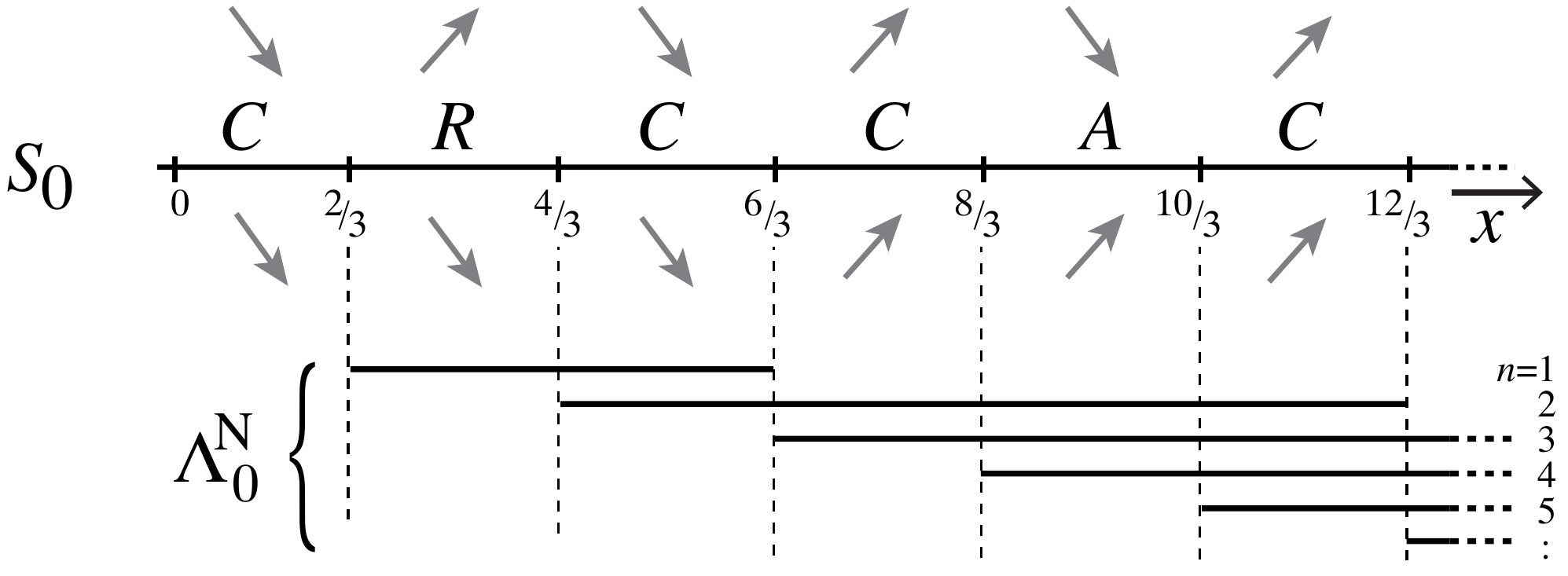}
\caption{The nonlinear switching system has sliding manifolds (branches of $\Lambda^{N}$) on $S_0$, represented here to show how they overlap increasingly with larger $x$.  }\label{fig:smnon}
\end{figure}

The sliding manifolds for $n=1$ and $n=2$ are depicted in \cref{fig:smnon}. It turns out that sliding may take place on $x\in(\sfrac23,\infty)$ irrespective of the directions of the vector field outside the switching threshold, i.e. not only through regions of attractive sliding, but also of repelling sliding and crossing. For $x>\sfrac43$ the sliding manifolds overlap as illustrated in \cref{fig:smnon}, and a regularization is needed to disclose on which manifold(s) a trajectory will evolve after hitting the switching threshold, and whether they can ever leave the switching threshold.

The sliding manifolds defined in \cref{snon} depends on $n$, such that a trajectory evolving on a specific manifold may slide for a maximum distance in $x$ of
\begin{align}
\Delta x (n)=2n-\sfrac{2n}{3}=\sfrac{4n}{3}\;.
\end{align}
This reveals a sort of \emph{ageing} in the system, as the `later' in $x$ a trajectory undergoes sliding motion, the `longer' sliding might persist. Such a situation is not found in the linear case, where all the sliding intervals have a constant length of $\sfrac{2}{3}$ (recall \cref{slin}).

From the direction of the vector fields in $S_\pm$ (see \cref{fig:sign}) we can discern that a sliding solution on $\Lambda^{N}$ must have 
\begin{subequations}
\begin{align}
&\mbox{entered from $S_+$ if $x \in \left(\sfrac{4n-2}{3},4n-2\right)$, or}\label{si+}\\
&\mbox{entered from $S_-$ if $x \in \left(\sfrac{4n}{3},4n\right)$, }\label{si-}
\end{align}
\end{subequations}
with $x$ denoting the time of entry.


\begin{lemma} \label{lempnl} For all $a \in \mathbb{R}^+$ we have that
\begin{align} \label{bff1} & P_+^a\left(4n-2\right) \in \left(4n-\sfrac{4}{3},4n-\sfrac{2}{3}\right), \ \forall  n \in \mathbb{N} \setminus\{0\}, \\
              \label{bff2}  & P_-^a\left(4n\right) \in \left(4n+2,4n+4\right), \ \forall  n \in \mathbb{N}.
\end{align}
\end{lemma}

\begin{proof} It is shown in \cref{lems1} that $P_+^a\left(\sfrac{10}{3}\right) \subset \left(4,\sfrac{14}{3}\right)$. The periodicity of the vector field in $S_+$ allows an immediate generalization of the result as follows
\begin{align}
 P_+^a\left(\sfrac{4n-2}{3}\right) \in \left(\sfrac{4n}{3},\sfrac{4n+2}{3}\right), \ \ n \in \mathbb{N} \setminus\{0\}\;,
\end{align}
which has relation \cref{bff1} as a particular case.
From this it is straightforward that, because of the periodicity of the system in $S_-$, it suffices to prove it for $n=0$. Recalling \cref{bart}-\cref{ppm}, let $\bar{x}_1=P_-^a\left(0\right)$ denote the first positive zero of
\begin{align}
 h_{-}\left(\bar{x}_1,0\right)=e^{-a\bar{x}_1}\sin \varphi_{-}\left(0\right)- \sin \left(\sfrac{\pi \bar{x}_1}{2}+\varphi_{-}\left(0\right)\right)=0, 
\end{align}
where \cref{varphi} and \cref{fik} yield
\begin{align}
 \varphi_{-}\left(0\right)=\frac{\pi \cdot 0 }{2}-\phi_-\in \left(-\sfrac{\pi}{2},0\right) \quad\Rightarrow \quad \sin \phi_-(0) < 0.
\end{align}
Therefore
\begin{align}
  h^0_{-}\left(\bar{x},0\right) < h_{-}\left(\bar{x},0\right) < h_{-}^{\infty}\left(\bar{x},0\right) < 0
\end{align}
in appropriate intervals, yielding $\bar{x}_\infty < \bar{x}_1 < \bar{x}_0$, where $\bar{x}_0,\bar{x}_\infty,$ are the first positive zeros of $h_{-}^0\left(\bar{x},0\right)$, $h_{-}^{\infty}\left(\bar{x},0\right)$, respectively.
Now from \cref{tha},
\begin{align}
 \left. \begin{array}{l} \bar{x}=4 \ \ (n=1), \\
                           \bar{x}= 2(2n+1)+\frac{4\phi_{-}}{\pi}-2\cdot 0 \in \left(2,4\right) \ \ (n=0) \end{array}
   \right\} \quad\Rightarrow \quad \bar{x}_0=4,
\end{align}
while \cref{thb} implies
\begin{align}
 \bar{x}=2n+\frac{2\phi_{-}}{\pi}-0 \in \left(2,3\right) \ \ (n=1) \quad\Rightarrow \quad \bar{x}_\infty=2\;.
\end{align}
Notice that the worst cases for the lower limit (lowest value for $\bar{x}_\infty$) and the upper limit (highest value for $\bar{x}_0$) have been selected. Then
\begin{align}
 2 < \bar{x}_1 < 4 \quad\Rightarrow \quad  x_1=\bar{x}_1-0=\bar{x}_1 \in \left(2,4\right), \quad \forall a \in \mathbb{R}^+\;.
\end{align}
\end{proof}

\begin{theorem} \label{inlnsps} System \cref{invnon} does not have non-sliding periodic solutions.
\end{theorem}

\begin{proof} Due to the overlapping of sliding intervals and the fact that $a>0$, non-sliding periodic solution candidates are solutions that cross $S_0$ on $x=2n$, $n \in \mathbb{N}$, and evolve alternatively in $S_+$ and $S_-$. However, it follows immediately from \cref{lempnl} that solutions exiting $S_0$ in such points hit again $y=0$ in less than two $x$-units in either case.
\end{proof}

Considering the possibility of jumping between sliding manifolds on $S_0$, it becomes immediately evident that there can exist infinitely many sliding periodic orbits. However, the ones we study below are particularly important, as we will see when we regularize.

\begin{theorem} \label{nlps} For all $a \in \mathbb{R}^+$, there exists $x_a \in (2,4)$ such that system \cref{invnon} has a unique sliding 4-periodic solution, $y_d(x)$, which lies in $\overline{S}_-$ and satisfies
\begin{align} & y_d\left(x\right) < 0, \ \ x \in \left(4n,4n+x_a\right), \\ 
               & y_d\left(x\right) = 0, \ \ x \in \left\{0\right\}\cup\left[4n+x_a,4(n+1)\right], \ \ n \in \mathbb{N}.
\end{align}							
\end{theorem}

\begin{proof} Let $y_d(0)=0$. Therefore $y_d(x)>0$ for all $x \in (0,x_a)$, with $x_a$ stemming immediately from \cref{bff2} in \cref{lempnl}: $x_a=P_-^a\left(0\right)-0 \in (2,4)$. Now, assume that $y_d(x)$ evolves on a sliding interval of the type \cref{si-} once on $S_0$, exiting at $x=4$. Then the result follows by periodicity.
\end{proof}

\begin{proposition} \label{nls} For all $a \in \mathbb{R}^+$, there exist solutions that start in $S_+$ and have their evolutions constrained in $\overline{S}_-$ for all $x > x_{T_a}$, $x_{T_a}$ denoting the smallest $x$ at which they hit the switching threshold. Moreover, they overlap with the periodic solution described in \cref{nlps} for $x \geq 4n$, with $n \in\mathbb{N}$.
\end{proposition}

\begin{proof} Any solution starting in $S_+$ hits $S_0$ in finite time, say $x=x_{T_a}$.  First let $x_{T_a} \in \left(0,\sfrac{2}{3}\right)$. As this is a crossing interval the system will continue its evolution in $S_-$, hitting again $S_0$ in $x > 2$. Assume that this solution, $y(x)$, follows a sliding interval of the type \cref{si-} once on $S_0$, then $y(4n)=0$ for a certain $n \in\mathbb{N}$, and thus intersects the periodic solution described in \cref{nlps}. The result follows assuming $y(x)$ to behave identically to such a periodic solution from the value of $x$ at which they intersect, which is allowed by the periodicity of the vector field. Now notice that $x_{T_a} \notin \left(\sfrac{2}{3},\sfrac{4}{3}\right)$, as this is a repelling sliding interval. Finally, let  $x_{T_a} > \sfrac{4}{3}$: assuming that $y(x)$ follows a sliding interval of the type \cref{si-} once on $S_0$, the result follows as in the case $x_{T_a} \in \left(0,\sfrac{2}{3}\right)$. 
\end{proof}

\Cref{fig:solsnl} illustrates the solutions described in \cref{nlps} and \cref{nls}.

\begin{remark} \label{therem} 
For this nonlinear system, because the sliding manifold $\Lambda^{N}$ is not 4-periodic in $x$ but instead consists of branches of ever increasing length, one must be careful in assuming that orbits that appear periodic for small $x$ will remain so. An orbit is simulated in \cref{fig:solsnl} that may appear $8$-periodic (and we note this partially overlaps the period 4 orbit). Concatenating solutions in $S_+$ and $S_-$ with branches of $\Lambda^{N}$ may suggest that this orbit is able to repeat every $\Delta x=8$, and yet to properly analyse the sliding dynamics requires close inspection of the dynamics on $S_0$ that can only be achieved by regularization.
When we do so in \cref{sec:rdnl} we shall see that exit from $S_0$ into $S_+$ becomes impossible for large $x$.  
\end{remark}

\begin{figure}[t]\centering
\includegraphics[width=0.8\textwidth]{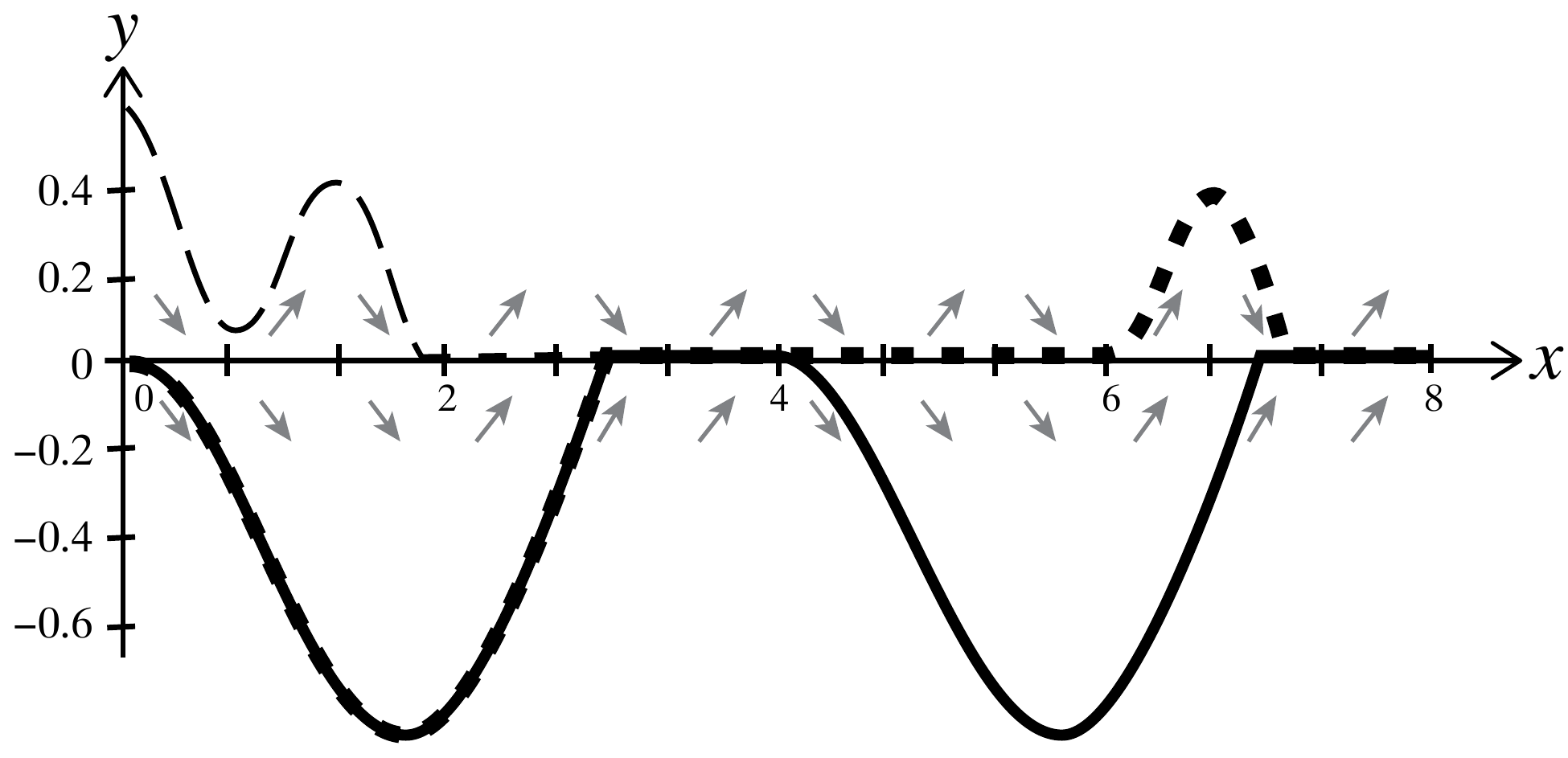}
\caption{Orbits of the nonlinear switching system, showing: a period 4 sliding solution (full curve, as described by \cref{nlps}) and a solution attracted onto it in finite time (dashed curve, as described by \cref{nls}), and a sliding solution (dotted curve) illustrating \cref{therem}) part of which coincides with the period 4 orbit and which appears over this short duration to be 8-periodic. All simulated for $a=0.5$.}\label{fig:solsnl}
\end{figure}



\section{Dynamics of a regularized system} \label{sec:rd}

The dynamics of our two switching systems is only partially resolved by the analysis in \cref{sec:idl} and \cref{sec:idnl}. To obtain a complete picture, particularly of solutions that involve sliding along the switching threshold, we must regularize the discontinuity. This involves replacing the switching threshold by a switching layer of width order $\epsilon > 0$, namely 
\begin{align}
S_0^\epsilon:=\left\{(x,y) \in \mathbb{R}^2\;:\; -\epsilon < y <  \epsilon \right\}\;.
\end{align}
Here we shall do this in a manner that yields a smooth system, replacing the switching multiplier $\lambda=\sign(y)$ by a transition function $\psi \left(\frac{y}{\epsilon}\right)$ satisfying
\begin{subequations} \label{psi}
\begin{align} 
              \label{psib} \psi ' \left(\frac{y}{\epsilon}\right) > 0 \ \ \mbox{for} \ \ \left|y\right|< \epsilon\;, \\
							\label{psic} \psi \left(\frac{y}{\epsilon}\right)=\mbox{sign}(y) \ \ \mbox{for} \ \ \left|y\right| \geq \epsilon\;, \\
							\label{psie} \mbox{sign}\left(\psi '' \left(\frac{y}{\epsilon}\right) \right) =  -\mbox{sign}(y) \ \ \mbox{for} \ \ \left|y\right| = \epsilon\;,
\end{align}
\end{subequations}
from which it follows that 
$|\psi (\frac{y}{\epsilon})|< 1$ for $|y|< \epsilon$ and 
$\psi ' (\frac{y}{\epsilon}) =  0$ for $|y| \geq \epsilon$. 

The half-planes corresponding to $S_\pm,$ defined in \cref{ulhp} for the discontinuous system, become
\begin{equation} \label{sreg} S^\epsilon_\pm:=\left\{(x,y) \in \mathbb{R}^2\;:\; \pm y \geq \epsilon \right\}.
\end{equation}

The regularization of the system \cref{s1d} is simply obtained by substituting $\psi(y/\eps)$ in place of $\lambda$. In terms of a fast variable $v={y}/{\epsilon}$ this becomes
\begin{subequations} \label{s1dr}
\begin{align}
\label{s1dra} \dot x&=1\;,
\\
\label{s1drb} \epsilon \dot v&=-a\epsilon  v-f_i\left(x,\psi \left(v\right)\right)\;,
\end{align}
\end{subequations}
where $f_i$ and $f_L$ are the functions defined in \cref{1f}. 

Re-scaling to a fast time variable $\tau=t/\epsilon$ yields the fast subsystem
\begin{subequations} \label{s1drf}
\begin{align}
x'&=\epsilon\;,
\\
v'&=-a\epsilon  v+f_i\left(x,\psi \left(v\right)\right)\;,
\end{align}
\end{subequations}
denoting the time derivative as $x'\equiv\eps\dot x$. The limit $\epsilon \rightarrow 0$ results in
\begin{subequations} \label{s1drff}
\begin{align}
x'&=0\;,
\\
v'&=f_i\left(x,\psi \left(v\right)\right)\;,\qquad\quad\;
\end{align}
\end{subequations}
which provides the fast dynamics along the $v$ direction, parameterized by $x$, outside the slow critical manifold
\begin{equation} \label{cm} \Lambda_0^{i}=\left\{\left(x,v\right): \ f_i\left(x,\psi \left(v\right)\right)=0, \ |v| < 1 \right\}\;,
\end{equation}
which is the equivalent of the sliding manifold $\Lambda^{i}$ of the discontinuous system defined in \cref{smi}, where again $i$ labels the linear or nonlinear model as in \cref{1f}; the subscript `$0$' denotes that this is assocated with the $\eps=0$ system i.e. the {\it critical manifold} of \cref{s1dr}, while the {\it slow manifolds} for $\eps>0$ will be denoted $ \Lambda_\eps^{i}$. The critical manifold is an invariant of the limiting system (i.e. with $\epsilon=0$) wherever it is normally hyperbolic, i.e. where
\begin{equation} \label{hy} \frac{\partial }{\partial v} f_i\left(x,\psi \left(v\right)\right) \neq 0,\end{equation}
and the positivity (or negativity) of this partial derivative indicates that $\Lambda_0^r$ is unstable (or stable).


Notice also that the analogues of $S_0^\epsilon$ and $S_\pm^\epsilon$ in the $(x,v)$ plane are
\begin{subequations} \label{vregs}
\begin{align} 
\label{vreg0} V_0:=&\left\{(x,v) \in \mathbb{R}^2\;:\; |v| <  1 \right\}, \\
\label{vreg} V_\pm:=&\left\{(x,v) \in \mathbb{R}^2\;:\; \pm v \geq 1 \right\}.
\end{align} 
\end{subequations}
(These are independent of $\eps$ because it has been scaled out by the $v$ coordinate). 

In turn, for $\epsilon \neq 0$ the fold points on $v=\pm 1$ lie at $O(\epsilon)$ distance from those of the limiting ($\eps=0$) system, namely, 
\begin{equation} \label{foldi} x_{n}^+:=\frac{2n}{3}, \ \ \ x_n^-:=2n.
\end{equation}
Indeed, it follows from \cref{s1drb} that, at the fold points, $x_{\epsilon,n}^\pm$,  one has
\begin{align}
 -a\epsilon \left(\pm 1\right)-\sin \left(\pi \omega_\pm x_{\epsilon,n}^\pm\right)=0,
\end{align}
where it is assumed that $a\epsilon <<1$, then 
\begin{equation} \label{foldr}
x_{\epsilon,n}^\pm:=x_n^\pm \pm \left(-1\right)^{n+1}\frac{ 1}{\pi \omega_\pm}\arcsin (a\epsilon).
\end{equation}
This is illustrated in \cref{fig:folds}.

\begin{figure}[t]\centering
\includegraphics[width=0.8\textwidth]{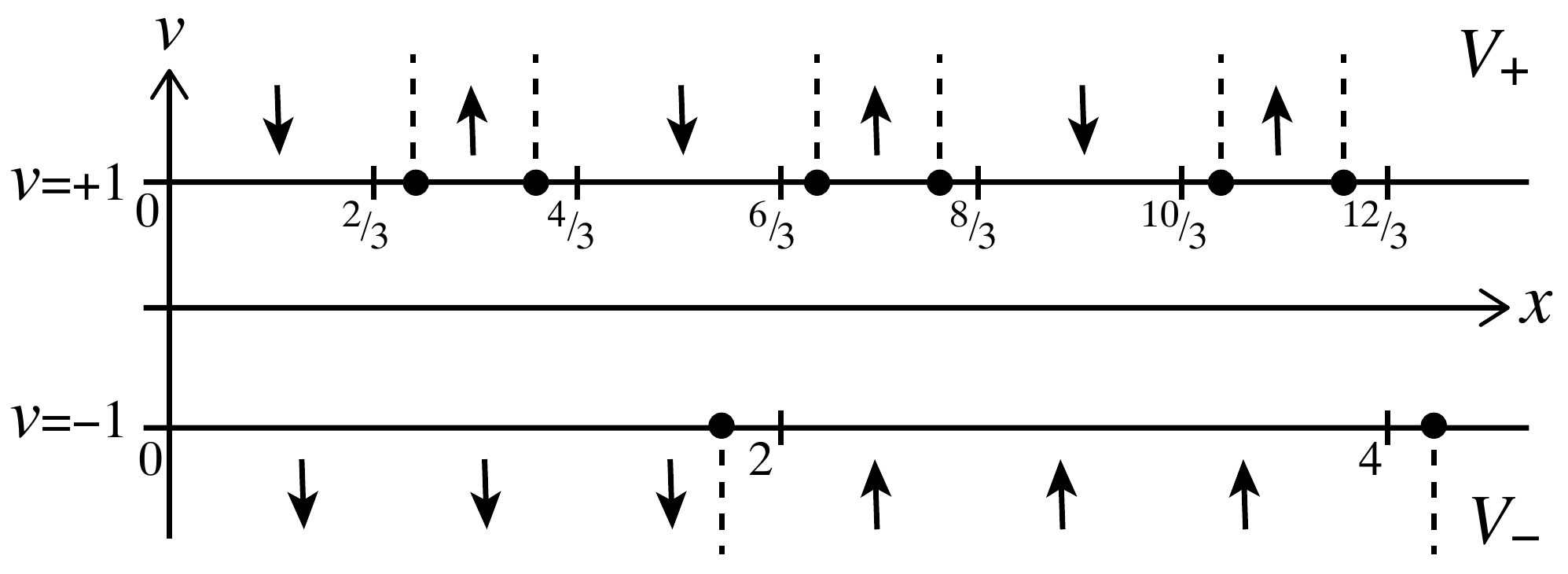}
\caption{Fold points (bold circles) and direction of the vector field for the regularized case.}\label{fig:folds}
\end{figure}

We show below that the qualitative dynamics of the linear switching system (using $f_L$ from \cref{1flin}) is preserved under regularization, including both sliding and non sliding periodic solutions. The qualitative dynamics of the nonlinear switching system (using $f_N$ from \cref{1fnon}) is also preserved, but instead of an exact period 4 solution, there exists an orbit asymptotically approaching period 4, with its segment in the layer $V_0$ deforming slightly with increasing $x$. Moreover we are able to show that under regularization:
\begin{itemize} 
\item[i)] ageing persists, such that branches of the sliding manifold not only increase in length with $x$ but do not overlap, so the system is not periodic inside the layer $V_0$ and fully periodic orbits are impossible; 
\item[ii)] for $x \rightarrow +\infty$ all solutions evolve towards a {\it periodic object} (see \cref{vr} and \cref{rempo}) which, for $\epsilon \rightarrow 0$, becomes the period 4 orbit of the discontinuous system described in \cref{nlps};
\item[iv)] solutions with any initial condition eventually lie in $V_- \cup V_0$ (similar to the behaviour illustrated by the full curve and dashed curve in \cref{fig:solsnl} but without periodicity). Hence period 8 orbits of the type portrayed by the dotted curve in \cref{fig:solsnl} are impossible.
\end{itemize}

As in previous sections we take the linear and nonlinear systems in turn. 


\subsection{Regularization of the linear switching system} \label{sec:rdl}

Take the system \cref{s1dr} using the forcing $f_L$ defined in \cref{1flin} (or equivalently substitute $\lambda\mapsto\psi(y/\eps)$ into \cref{invlin}). We obtain
\begin{subequations} \label{lr}
\begin{align}
\label{lra} \dot x&=1,
\\
\label{lrb} \epsilon \dot v&=-a\epsilon v -\left[1+\left(1+\psi\left(v\right)\right)\cos \pi x \right]\sin \sfrac{\pi x}{2}.
\end{align}
\end{subequations}

\begin{proposition} The critical manifolds of \cref{lr} (in the limit $\epsilon = 0$) are given by
\begin{equation} \label{lrsm} \Lambda_0^{L}:=\left\{(x,v) \in V_0\;: \; \psi\left(v\right)= -1-{\sec \pi x}, \ x \in \left(\sfrac{2}{3}+2n,\sfrac{4}{3}+2n\right), \ n \in \mathbb{N}\right\}.
\end{equation}
\end{proposition}

\begin{proof} It follows from \cref{cm} that the critical manifolds are defined by
\begin{align}
  \left[1+\left(1+\psi \left(v\right)\right)\cos \pi x\right]\sin\sfrac{\pi x}{2}=0, \quad |v|<1\;.
\end{align}
Therefore $-1 < \lambda=2\left(\frac{n}{x}-1\right) < 1$, implying $x \in \left(\frac{2n}{3},2n\right)$ for $n \in \mathbb{N}\setminus\{0\}$. However, by \cref{psi} the function $\psi$ satisfies $\left|\psi \left(v\right)\right|<1$, therefore
\begin{align}
 -1 < \psi \left(v\right)= -1-{\sec \pi x} < 1\;,
\end{align}
hence $0 < -\frac{1}{\cos \pi x} < 2$, implying $\cos \pi x < -\sfrac{1}{2}$ and therefore
\begin{align}
 x \in \left(\sfrac{2}{3}+2n,\sfrac{4}{3}+2n\right), \ \ n \in \mathbb{N}\;.
\end{align}
As in defining $\Lambda^{L}$ we may exclude the point where $\sin\sfrac{\pi x}{2}=0$, hence where $x=2n$ for $n \in \mathbb{N}$. 
Following \cref{hy}, $\Lambda_0^{L}$ is invariant where
\begin{equation} \label{hyl} 0 \neq  \frac{\partial }{\partial v} f\left(x,\psi \left(v\right)\right) = - \psi ' \left(v\right){\cos \pi x}\sin\sfrac{\pi x}{2}\;,\end{equation}
and since $\psi ' \left(v\right) >0$ for all $v \in (-1,1)$ by \cref{psib}, the result follows.
\end{proof}

Finally, from the sign of $\frac{\partial }{\partial v} f\left(x,\psi \left(v\right)\right)$ in \cref{hyl} we have that
\begin{subequations} \label{lrsmra}
\begin{align} 
&{\rm on}\;\; x \in  \left(\sfrac{2}{3}+4n,\sfrac{4}{3}+4n\right) 
\;\; \Lambda_0^{L} \;\; \mbox{is repelling}, \\ 
&{\rm on}\;\; x \in  \left(\frac{8}{3}+4n,\sfrac{10}{3}+4n\right) 
\;\; \Lambda_0^{L} \;\; \mbox{is attracting},
\end{align}
\end{subequations}
with $n \in\mathbb N$. These $x$-intervals match those obtained for $\Lambda^{L}$ in the discontinuous linear system in \cref{sec:idl} as illustrated in \cref{fig:sign}.

As an example, the critical manifolds corresponding to the $x$-intervals defined in \cref{lrsmra} for $n=1$, letting
\begin{equation} \label{qg}
\psi \left(v\right)=\left\{\begin{array}{cc} -1 & v < -1\;, \\ \sfrac{1}{2}v\left(3-v^2\right)  & -1 \leq v \leq 1\;, \\ 1 & v > 1 \;,\end{array} \right.
\end{equation}
are depicted in \cref{fig:lrsm}.

\begin{figure}[t]\centering
\includegraphics[width=0.8\textwidth]{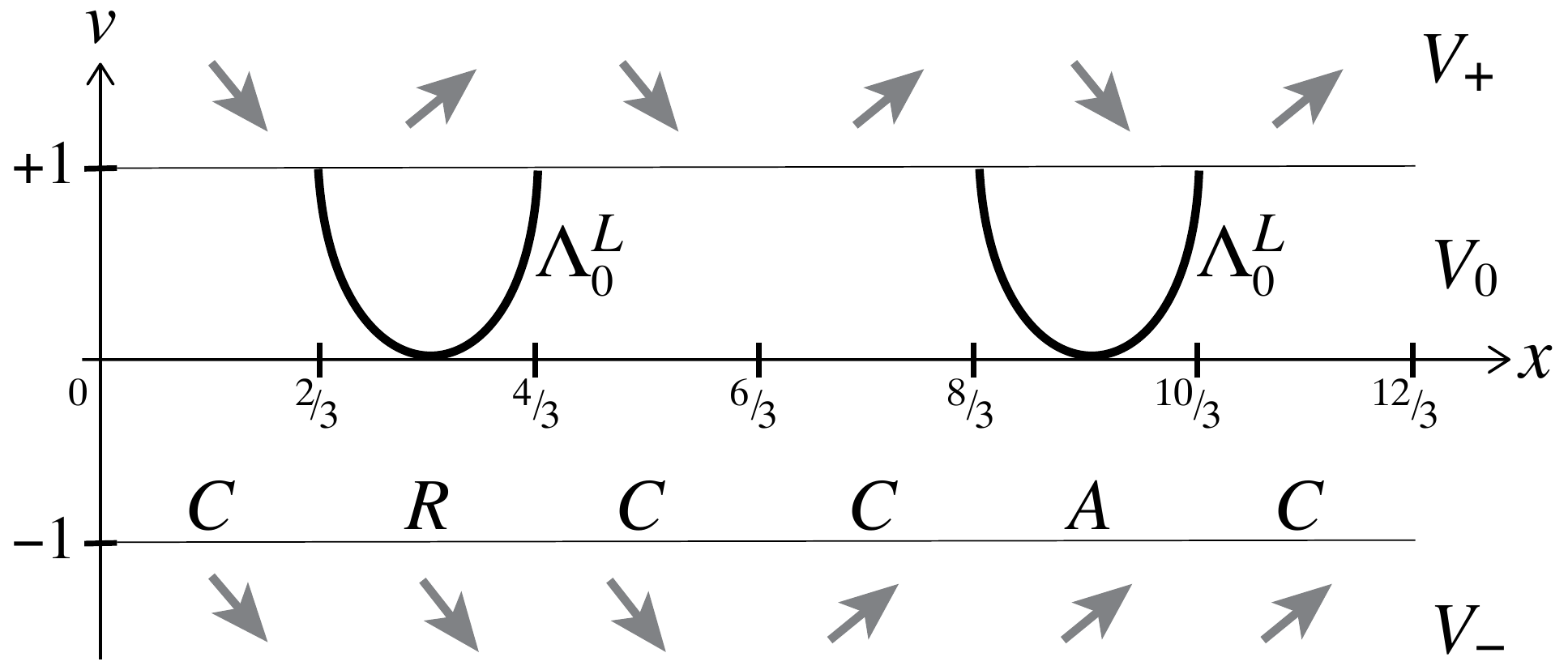}
\caption{Linear regularization, showing the switching layer $V_0$ and the sliding manifolds $\Lambda_0^{L}$ inside it, with a repelling branch (left) and attracting branch (right). Corresponding labels for regions of crossing and sliding from \cref{fig:sign} are shown, with the directions of the fields in $V_\pm$.}\label{fig:lrsm}
\end{figure} 

The persistence of the non-sliding and sliding periodic solutions are established below by analogous results to those of \cref{teo1} and \cref{teo2}.

\begin{theorem} \label{teo3} For all $a,\epsilon \in \mathbb{R}$ such that $0 < a\ll1$ and $0<\epsilon \ll 1$, system \cref{lr} has a non-sliding periodic solution, which is $4$-periodic and locally asymptotically stable.
\end{theorem}

\begin{proof} 
In the discontinuous system the existence of a periodic orbit, given in \cref{teo1}, was proven in \cref{lem2} and \cref{lem3} using the contractivity of a Poincar\'e map $P(x,a)=P_+^a\left(P_-^a(x)\right)$ defined in \cref{Pdisc} for $0 < a \ll 1$. For the regularized system we define an analogous map $P_\epsilon(x,a)=\left(P_{\epsilon +}^a \circ P_{\epsilon -}^a\right)(x)$. 
We may decompose $P_-^a$ and $P_{\epsilon -}^a$ each into three components 
\begin{align}
  P_{\epsilon -}^a:=  P_{\epsilon -,3}^a \circ P_{\epsilon -,2}^a \circ P_{\epsilon -,1}^a\;, \quad P_{-}^a:=  P_{-,3}^a \circ P_{-,2}^a \circ P_{-,1}^a\;,
\end{align}
as illustrated in \cref{fig:maps2}. 
It is immediate that for $\epsilon \rightarrow 0$ we have $P_{\epsilon -,k}^a,P_{-,k}^a \rightarrow \mathbb{I}$ for $k=1,3$, while $P_{\epsilon -,2}^a=P_{-,2}^a$, for all $\epsilon$. Standard results on the regularity of initial conditions and transversality guarantee that $P_{\epsilon -}^ a \stackrel{\mathcal{C}^1}{\rightarrow} P_-^a$  for $\epsilon \rightarrow 0$. Defining similar component maps for $P_+^a$ and $P_{\epsilon +}^a$ we have also $P_{\epsilon +}^a \stackrel{\mathcal{C}^1}{\rightarrow} P_+^a$, which implies  $P_{\epsilon} \stackrel{\mathcal{C}^1}{\rightarrow} P$ for $\epsilon \rightarrow 0$. The result then follows immediately from \cref{teo1}.
\end{proof}

\begin{figure}[t]\centering
\includegraphics[width=0.7\textwidth]{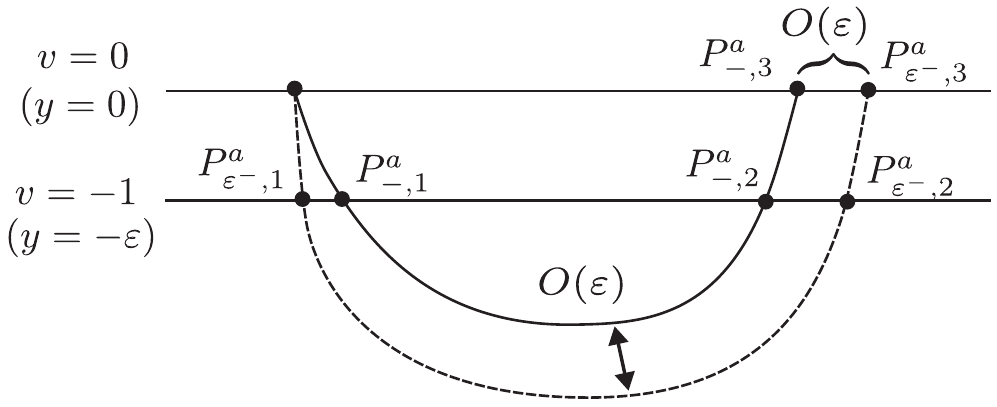}
\caption{The component mappings $P_{\epsilon -,k}^a$ and $P_{-,k}^a$ for $k=1,2,3$.}\label{fig:maps2}
\end{figure}

\begin{theorem} \label{teo4} For all $a,\epsilon >0$ such that $a \gg 1$ and $\epsilon \ll  1$, system \cref{lr} has a sliding $4$-periodic solution. Such a solution is locally asymptotically stable, with Lipschitz constant exponentially small in $\epsilon$.
\end{theorem}

\begin{proof} According to \cref{teo2}, sliding periodic solutions in the discontinuous case have crossing and sliding sections.
Following the proof of \cref{teo3}, the Poincar\'e maps $P$ and $P_\epsilon$ can be defined and expressed as a composition of a finite number of successive maps, each from crossing to crossing intervals, from crossing to sliding, or from sliding to crossing intervals. For the maps from crossing to crossing intervals, the discussion in the proof of \cref{teo3} applies, which means that they are Lipschitz continuous, with Lipschitz constant independent of $\epsilon$. For the crossing$\leftrightarrow$sliding maps, theorem  2.1 in \cite{bonet2016} applies directly, thus guaranteeing that the Lipschitz constant is exponentially small in $\epsilon$.
\end{proof}

In \cref{fig:lr1} we simulate a $4$-periodic non-sliding solution for $a=0.01$, and a $4$-periodic sliding periodic solution for $a=2$. The two lower panels show a magnification of the trajectories passing through the switching layer. Notice the equivalence with the solutions plotted in \cref{fig:a=2} from the discontinuous system.

\begin{figure}[t]\centering
\includegraphics[width=0.8\textwidth]{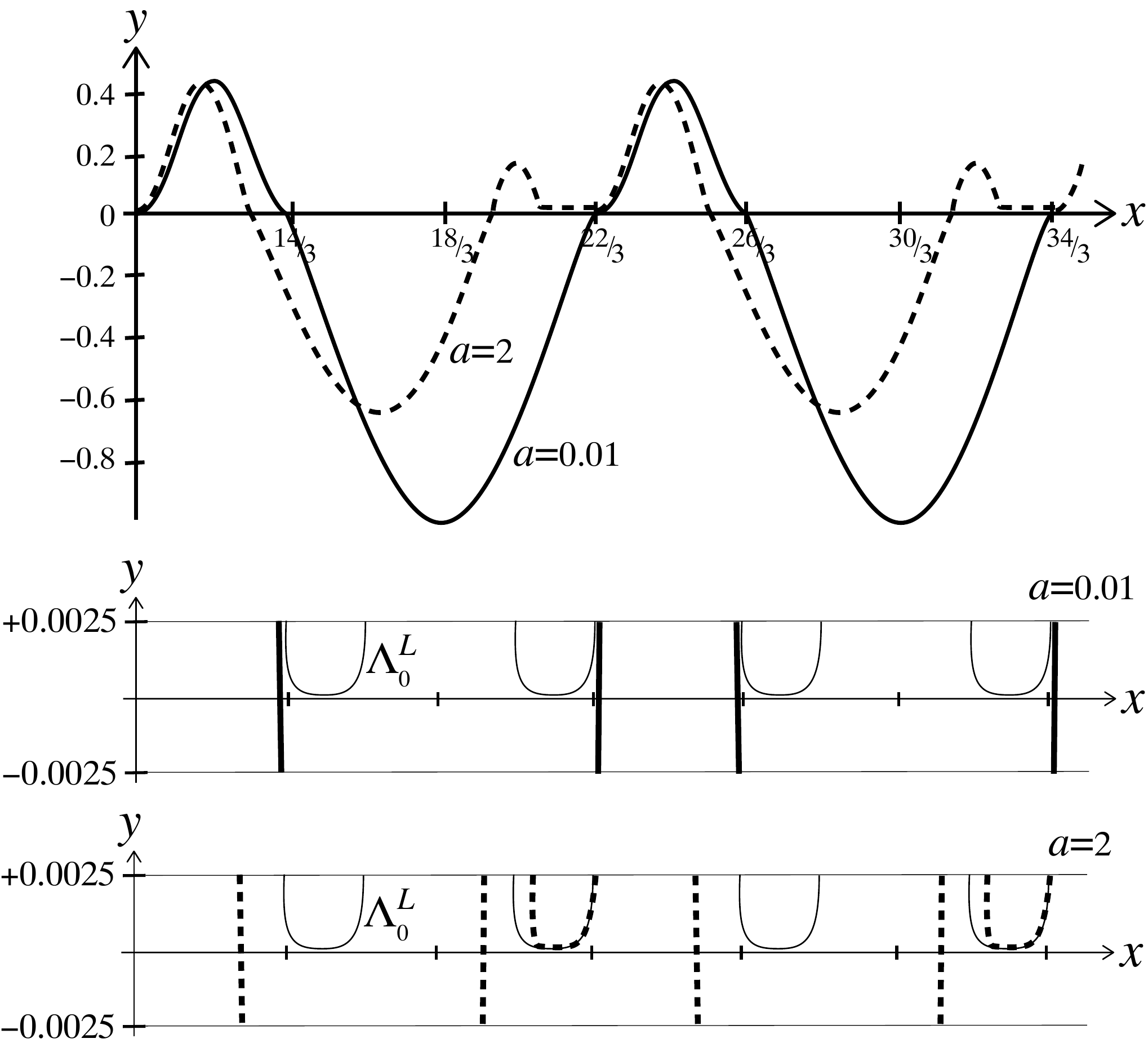} 
\caption{Linear regularization, showing a non-sliding period 4 solution (full curve) for $a=0.01$ passing through the point $\left(x_0,y_0\right)=\left(0,0.42
\right)$, and a sliding period 4 solution (dotted curve) for $a=2$ passing through the point $\left(x_0,v_0\right)=\left(-\sfrac{2}{3},0.0019
\right)$; both simulated from \cref{lr} with $\epsilon=0.0025$. The lower panels are magnifications of the switching layer, showing the solution passing through the layer, and in the latter case evolving close to the critical manifold $\Lambda_0^{L}$.}\label{fig:lr1}
\end{figure}


%


\subsection{Regularization of the nonlinear switching system} \label{sec:rdnl}

Take the system \cref{s1dr} using the forcing $f_N$ defined in \cref{1fnon} (or equivalently substitute $\lambda\mapsto\psi(y/\eps)$ into \cref{invnon}). We obtain
\begin{subequations} \label{nlr}
\begin{align}
\label{nlra} \dot x&=1,
\\
\label{nlrb} \epsilon \dot v&=-a\epsilon v -\sin \left(\pi x \left[1+\sfrac{1}{2}\psi \left(v\right)\right]\right).
\end{align}
\end{subequations}

\begin{proposition} The critical manifolds of \cref{nlr} (in the limit $\epsilon = 0$) are given by
\begin{equation} \label{snl} \Lambda_0^{N}:=\left\{(x,v) \in V_0: \ \psi \left(v\right) = 2\left(\sfrac{n}{x}-1\right), \ x \in \left(\sfrac{2n}{3},2n\right), \ n \in \mathbb{N}\setminus\{0\}\right\}.
\end{equation}
\end{proposition}

\begin{proof} It is straightforward from \cref{cm} that the critical manifolds are defined by
\begin{align}
  \sin \left(\pi x \left[1+\sfrac{1}{2}\psi \left(v\right)\right]\right)=0, \quad |v|<1\;.
\end{align}
The result follows proceeding as in the linear case, noting the manifolds are invariant everywhere, since
\begin{align} \nonumber 
\frac{\partial }{\partial v} f_N\left(x,\psi \left(v\right)\right) & = - \sfrac{\pi x}{2}\psi ' \left(v\right)\cos \left(\pi x \left[1+\sfrac{1}{2}\psi \left(v\right)\right]\right) \\ \label{hynl} &=- \sfrac{\pi x}{2}\psi ' \left(v\right)\cos n\pi= \frac{(-1)^{n+1}\pi x}{2}\psi ' \left(v\right) \neq 0.\end{align}
\end{proof}

It then follows immediately from \cref{hynl} that
\begin{subequations}
\begin{align} 
& \Lambda_0^{N} \;\; \mbox{is repelling}\;\;\;{\rm on}\;\;x \in \left(\frac{4n-2}{3},4n-2\right)\;,\label{rsnl}  \\
& \Lambda_0^{N} \;\; \mbox{is attracting}\;\;\;{\rm on}\;\;x \in \left(\frac{4n}{3},4n\right)\;,\label{asnl} 
\end{align}
\end{subequations}
with $n \geq 1$. 

Examples of the critical manifolds corresponding to $n=1,2 \ldots ,10,$ are shown in \cref{fig:nlrsm_ar}, again using \cref{qg}.

\begin{figure}[t]\centering
\includegraphics[width=0.8\textwidth]{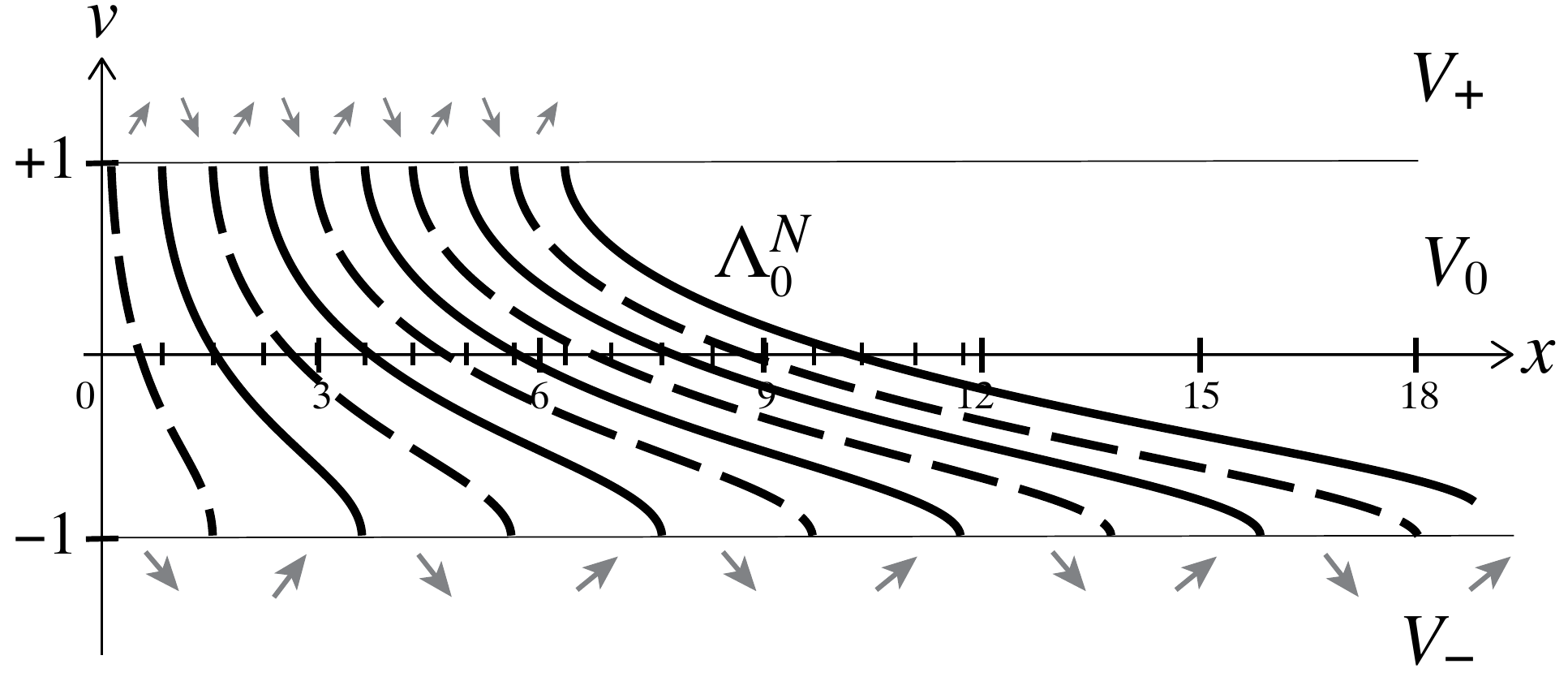}
\caption{Nonlinear regularization, showing the switching layer $V_0$ and the sliding manifolds $\Lambda_0^{N}$ inside it, with repelling branches (dashed) and attracting branches (full), for $n=1,2 \ldots 10$. Arrows indicate the direction of the vector fields in $S_\pm$.}\label{fig:nlrsm_ar}
\end{figure} 

Notice that sliding in the nonlinear system may take place irrespective of the directions of the vector fields outside the layer, due to the presence of slow invariant manifolds, making sense of the proliferation of branches of sliding in the  discontinuous system. The regularization shows that the slow critical manifolds corresponding to sliding do not intersect, but they do overlap in the sense that multiple branches co-exist at certain intervals of $x$ values. As in the discontinuous system, the union of all the sliding $x$-intervals covers $\mathbb{R}^+ \setminus \{\left(0,\sfrac{2}{3}\right)\}$, and the $n^{th}$ slow critical manifold exists over a range in $x$ of width
\begin{align}
\Delta x (n)=2n-\frac{2n}{3}=\frac{4n}{3}\;.
\end{align}
Thus the \emph{ageing} phenomenon as described in \cref{sec:idnl} for the discontinuous nonlinear switching system persists.

Notice also that:

\begin{itemize} \item [i)] Trajectories from $V_\pm$ enter the switching layer through the $x$-intervals where the corresponding vector field points towards $V_0$. Following \cref{foldr}, these are
\begin{subequations} \label{foldint}
\begin{align} 
\label{foldint+} & x \in \left(x_{\epsilon,2n}^+,x_{\epsilon,2n+1}^+\right) \ \ \mbox{from} \ \ V_+ \\
\label{foldint-} & x \in \left(x_{\epsilon,2n-1}^-,x_{\epsilon,2n}^-\right) \ \ \mbox{from} \ \ V_-.
\end{align} 
\end{subequations}
Hence, $V_0$ is accessed from $V_+$  between a stable and an unstable slow manifold, and from $V_-$ between an unstable and a stable manifold. The arrows in \cref{fig:folds} indicate the direction of the vector fields in $V_\pm$.

\item[ii)] Within the switching layer and for $x > \sfrac{2}{3}$, trajectories are eventually attracted exponentially towards a slow invariant manifold \cite{f79} that lies $\frac{\epsilon}{n}$-close (except in a neighbourhood of the upper and lower fold points, see \cite{bonet2016}, \cref{below1} and \cref{ufp}), to the corresponding stable critical manifold $\Lambda_0^{N}$,  
see \cref{fig:nlrsm_ar}).  Once trapped, they always move towards $V_-$ pushed by the vector field, and never get back to $V_+$. It is worth emphasizing that the $\frac{\epsilon}{n}$-closeness is consistent with the fact that the ageing phenomenon makes the 'density' of critical manifolds within the switching layer at a certain $x$-instant to be of order $n$.

\item[iii)] As pointed out at the beginning of Section \ref{sec:rd}, when $x \rightarrow +\infty$ the trajectories of \cref{nlr} tend to a periodic object (recall Definition \ref{vr} and Remark \ref{rempo}) that, for $\epsilon \rightarrow 0$, constitutes the periodic solution of the discontinuous system described in \cref{nlps}.

\end{itemize}

The above statements are supported as follows. 

Let us first establish a result that will be used in the analysis of the dynamics of \cref{nlr}.  Let $v_\pm\left(x,x_0^\pm\right)$ denote the solutions of \cref{nlr} with initial conditions $\left(x_0^\pm,\pm 1\right)$. It is immediate that $v_\pm$ evolve on $V_\pm$ for $x_0^+=x_{\epsilon,2n+1}^+$, $x_0^-=x_{\epsilon,2n}^-$, respectively, and eventually reach $V_0$ in finite time because of the fact that $a>0$. It then makes sense, in the spirit of the map $P_\pm^a$ introduced in \cref{ppm}, to define $\bar{P}_{\epsilon,\pm}^a$ as the map providing the next intersection point of $v_\pm$ with $v=\pm 1$, respectively.  


\begin{lemma} \label{below2}  For all $n \in \mathbb{N}$, we have
\begin{subequations} \label{entering}
\begin{eqnarray} 
\label{entering+} \bar{P}^a_{\epsilon,+}\left(x_{\epsilon,2n+1}^+\right) & \in & \left(x_{\epsilon,2n+2}^+,x_{\epsilon,2n+3}^+\right) \\
\label{entering-} \bar{P}^a_{\epsilon,-}\left(x_{\epsilon,2n}^-\right) & \in & \left(x_{\epsilon,2n+1}^-,x_{\epsilon,2n+2}^-\right).
\end{eqnarray} 
\end{subequations}
\end{lemma}

\begin{proof} 
At $v=1$, the vector field points towards $V_+$ on $\left(x_{\epsilon,2n+1}^+,x_{\epsilon,2n+2}^+\right)$ and towards $V_0$ on $\left(x_{\epsilon,2n+2}^+,x_{\epsilon,2n+3}^+\right)$, therefore $\bar{P}^a_{\epsilon,+}\left(x_{\epsilon,2n+1}^+\right) > x_{\epsilon,2n+2}^+$. An analogous argument at $v=-1$ implies $\bar{P}^a_{\epsilon,-}\left(x_{\epsilon,2n}^-\right) > x_{\epsilon,2n+1}^-$. 

Now take initial states $x_0^+=x_{\epsilon,2n+1}^+$ and $x_0^-=x_{\epsilon,2n}^-$. The solutions of \cref{nlr} with initial conditions $\left(x_0^\pm,\pm 1\right)$ are
\begin{align} 
v_\pm\left(x,x_0^\pm\right)=&\frac{\omega_\pm \pi \cos \left(\omega_\pm\pi x\right) - a \sin \left(\omega_\pm\pi x\right)}{\epsilon \left(\omega_\pm^2\pi^2+a^2\right)} \nonumber\\ &\pm e^{-a\left(x-x_0^\pm\right)}\left(1+\frac{\omega_\pm\pi\sqrt{1-a^2\epsilon^2}-a^2\epsilon}{\epsilon \left(\omega_\pm^2\pi^2+a^2\right)}\right).
\end{align}
Letting $x_a^+=x-x_0^+$, $x_a^-=x-x_0^-$, the conditions $v_\pm\left(x,x_0^\pm\right)=\pm 1$ result in 
\begin{align}
  v_\pm\left(x,x_0^\pm \right)-\left(\pm 1\right)=\frac{\pm 1}{\omega_\pm^2\pi^2+a^2}h_\pm\left(x_a^\pm\right),
\end{align}
with
\begin{align}  h_\pm\left(x_a^\pm\right):=&a\left(\epsilon\omega_\pm\pi+\sqrt{1-a^2\epsilon^2}\right) \sin \left(\omega_\pm\pi x_a^\pm\right) \nonumber\\ &-\left(\omega_\pm\pi\sqrt{1-a^2\epsilon^2}-a^2\epsilon\right) \cos \left(\omega_\pm\pi x_a^\pm\right) - \epsilon \left(\omega_\pm^2\pi^2+a^2\right) \nonumber\\ &+ e^{-ax_a^\pm}\left(\epsilon \omega_\pm^2\pi^2+\omega_\pm\pi\sqrt{1-a^2\epsilon^2}\right).
\end{align}
It is clear that, for all $a,\epsilon \in \mathbb{R}^+$, with $a\epsilon <<1$,
\begin{align}
 h_u^\pm\left(x_a^\pm\right) > h_\pm\left(x_a^\pm\right),
\end{align} 
where
\begin{align}  h_u^\pm\left(z\right)&:=a\left(\epsilon\omega_\pm\pi+\sqrt{1-a^2\epsilon^2}\right) \sin \left(\omega_\pm\pi z\right) \nonumber\\ &\qquad+\left(\omega_\pm\pi\sqrt{1-a^2\epsilon^2}-a^2\epsilon\right) \left[1-\cos \left(\omega_\pm\pi z\right)\right].
\end{align}			
As, by construction, $h_\pm\left(z\right)>0$ when $z \rightarrow 0^+$, its first zero after $z=0$ is upper bounded by that of $h_u^\pm\left(z\right)$. Then it is immediate that
\begin{align}
h_u^\pm\left(z\right)=0 \ \Longleftrightarrow \  z_\pm=\frac{2m}{\omega_\pm}, \ m \in \mathbb{Z}. 
\end{align}
For $m=1$ this gives
\begin{align}
 x_a^\pm=x^\pm-x_0^\pm < \frac{2}{\omega_\pm}
\end{align}
and, recalling \cref{foldr}, one obtains 
\begin{align} & V_+: \ x^+<x_{\epsilon,2n+1}^++\frac{4}{3}=x_{\epsilon,2n+3}^+, \\
 & V_-: \ x^-<x_{\epsilon,2n+1}^-+4=x_{\epsilon,2n+2}^-.
\end{align}
Hence, the result follows.
\end{proof}

Let us consider the stable critical manifold
\begin{equation} \label{ssmn} \Lambda^{N}_{0n}:=\left\{(x,v) \in V_0: \ \psi \left(v\right) = 2\left(\frac{2n}{x}-1\right), \ x \in \left(\frac{4n}{3},4n\right)\right\}, 
\end{equation}
and let  $\delta \in \mathbb{R}^+$ be such that
\begin{equation} \label{delta} 1 \gg \delta > x^+_{2n}-x_{\epsilon,2n}^+=\frac{2}{3\pi}\arcsin (a\epsilon).
\end{equation}
Notice that $\Lambda^{N}_{0n}$ is enveloped within the manifolds
\begin{align}
 \Lambda^{N}_{\pm\delta n}:=\left\{(x,v) \in \mathbb{R}^2: \ \psi \left(v\right) = \frac{4n\pm 3 \delta}{x}-2, \ x \in \left[\frac{4n}{3}\pm \delta,4n\pm 3\delta\right]\right\}.
\end{align}
(These are a $\delta$-perturbation of $\Lambda_0^N$ but are {\it not} invariant manifolds). 
Setting 
\begin{align}
 \bar{\delta}= \sin \frac{3\pi\delta}{2}>a\epsilon,
\end{align}
it is immediate that
\begin{align}
 \sin \left(\pi x \left[1+\sfrac{1}{2}\psi \left(v\right)\right]\right)=\left\{\begin{array}{cc} 0 & \forall \ (x,v) \in \Lambda^{N}_{0n} \\ \pm \bar{\delta} & \forall \ (x,v) \in \Lambda^{N}_{\pm \delta n}.   \end{array}\right. 
\end{align}
Now, let $\mathcal{K}_{\delta n}$ be the compact set defined as
\begin{align} \mathcal{K}_{\delta n}:=&\left\{(x,v) \in \mathbb{R}^2: \ x \in \left[\frac{4n}{3}+\delta,\frac{4n+2}{3}-\delta\right] \ \wedge \ \left|v\right| \leq 1 \  \right. \nonumber\\ & \left. \wedge \ \psi \left(v\right) \geq \frac{4n+ 3 \delta}{x}-2\right\}\;.
\end{align}
Notice that, for all $(x,v) \in \mathcal{K}_{\delta n}$ and $\epsilon \rightarrow 0$, \cref{nlrb} is such that
\begin{align}
 \dot v=-a v -\frac{1}{\epsilon}\sin \left(\pi x \left[1+\sfrac{1}{2}\psi \left(v\right)\right]\right) \leq a-\frac{1}{\epsilon}\bar{\delta} \xlongrightarrow[\epsilon \rightarrow 0]{}-\infty.
\end{align}
This means that the  regular form of the vector field \cref{nlr},  i.e. \cref{s1drf},  is pointing in the negative direction of the $v$ axis with infinite modulus, so $\mathcal{K}_{\delta n} \cap \Lambda^{N}_{\delta n}$ is reached in finite time. Then by regularity with respect to initial conditions and parameters, there exists $\epsilon_\delta=\epsilon_\delta \left(\delta\right) \in \mathbb{R}^+$ such that, for all $\epsilon  \in \left(0,\epsilon_\delta\right)$, any trajectory starting within the subcompact set $\mathcal{K}_{2\delta n} \subset\mathcal{K}_{\delta n}$, where 
\begin{align} \mathcal{K}_{2\delta n}:=&\left\{(x,v) \in \mathbb{R}^2: \ x \in \left[\frac{4n}{3}+2\delta,\frac{4n+2}{3}-2\delta\right] \ \wedge \ \left|v\right| \leq 1 \  \right. \nonumber\\ & \left. \wedge \ \psi \left(v\right) \geq \frac{4n+ 3 \delta}{x}-2\right\},
\end{align}
evolves in $\mathcal{K}_{\delta n}$ and leaves it through 
the (upper) border of the compact set
\begin{align} \nonumber \mathcal{V}_{2\delta n}:=&\left\{(x,v) \in \mathbb{R}^2: \ x \in \left[\frac{4n}{3}+2\delta,4n-6\delta\right] \ \wedge \ \left|v\right| \leq 1 \  \right. \\ \label{s2dn} & \left. \wedge \ \frac{4n - 3 \delta}{x}-2 \leq \psi \left(v\right) \leq \frac{4n+ 3 \delta}{x}-2\right\},
\end{align}
where it arrives in finite time.

It follows from \cref{psi} and the compactness of $\mathcal{V}_{2\delta n}$ that $\psi'(v)$ has a positive lower bound in  $\mathcal{V}_{2\delta n}$, namely, $L_{\psi'}=L_{\psi'}\left(\delta,n\right)$ such that $0<L_{\psi'}\leq \psi'(v)$. Hence, there exists $\epsilon_\pm \in \mathbb{R}^+$ such that, for all $(x,v) \in \mathcal{V}_{2\delta n} \cap \Lambda^{N}_{\delta n}$  and $\epsilon \in \left(0,\epsilon_+\right)$ the flow is pointing downwards, i.e.
\begin{align} \vec{\nabla}\Lambda^{N}_{\delta n} \cdot \left(\dot x, \dot y\right)=&1+\frac{\psi \left(v\right)}{2}-\frac{x\psi' \left(v\right)}{2\epsilon}
\left(a \epsilon v +\sin \left(\pi x \left[1+\sfrac{1}{2}\psi \left(v\right)\right]\right)\right) \leq \frac{3}{2}\nonumber\\ & -\frac{2n+3\delta}{3\epsilon}\psi' \left(v\right)\left(\bar{\delta}-a\epsilon \right)\leq \frac{3}{2}-\frac{2+3\delta}{3\epsilon}L_{\psi'}\left(\bar{\delta}-a\epsilon \right) \ < 0,
\end{align}
while a similar procedure implies that the flow is pointing upwards for all $(x,v) \in \mathcal{V}_{2\delta n} \cap \Lambda^{N}_{-\delta n}$  and $\epsilon \in \left(0,\bar{\epsilon}_-\right)$, with $\epsilon_-=\frac{\bar{\delta}}{a}$, hence
\begin{equation} \label{lam-del} \vec{\nabla}\Lambda^{N}_{-\delta n} \cdot \left(\dot x, \dot y\right)> 1+\frac{2+3\delta}{3\epsilon}L_{\psi'}\left(\bar{\delta}-a\epsilon \right)>0.
\end{equation}
Then for all $\epsilon \in \left(0,\min\{\epsilon_{\delta},\epsilon_+,\epsilon_-\}\right)$, any trajectory entering $\mathcal{V}_{2\delta n}$ evolves within this set and leaves it through the right border, namely $\mathcal{V}_{2\delta n} \cap \{(x,v) \in \mathbb{R}^2: \ x=4n-6\delta\}$. 

Recall now \cref{foldi} and \cref{foldr}. As regards trajectories entering $V_0$ from $V_+$ by $\left(x_{\epsilon,2n}^+,x^+_{2n}+2\delta\right)$, the vertical vector field component on $v=1$ is negative, i.e. $\dot v < 0$, and $\dot x=1$. Moreover, the vector field on $\Lambda^{N}_{-\delta n}$ is pointing upwards also for $x \in \left(x_{\epsilon,2n}^+,x^+_{2n}+2\delta\right)$ and $\epsilon \in \left(0,\epsilon_-\right)$, as \cref{lam-del} holds in this region as well.
Therefore, the trajectories within
\begin{align} \mathcal{\bar{K}}_{2\delta n}:=&\left\{(x,v) \in \mathbb{R}^2: \ x \in \left(x_{\epsilon,2n}^+,x^+_{2n}+2\delta\right) \ \wedge \ \left|v\right| \leq 1 \ \wedge \right.\nonumber \\ & \left. \wedge \ \psi \left(v\right) \geq \frac{4n- 3 \delta}{x}-2\right\},
\end{align}
are directed towards the left border of $\mathcal{K}_{2\delta n} \cup \mathcal{V}_{2\delta n}$, which is reached in finite time.

The above discussion implies the following Lemma.

\begin{lemma} \label{lemant} Let $a \in \mathbb{R}^+$ and $1\gg\delta>0$. Then there exist $\epsilon_\delta,\epsilon_+,\epsilon_- \in \mathbb{R}^+$ such that, $\forall n \geq 1$ and $\epsilon \in \left(0,\bar{\epsilon}_+\right)$, with $\bar{\epsilon}_+=\min\{\epsilon_\delta,\epsilon_+,\epsilon_-\}$, the trajectories of \cref{nlr} entering the switching layer $V_0$ from $V_+$ through $x \in \left(x_{\epsilon,2n}^+,x^+_{2n+1}-2\delta\right)$ reach $\mathcal{V}_{2\delta n}$ in finite time, and keep evolving therein until they leave it by $x \geq 4n-6\delta$. \qed
\end{lemma}  

\begin{remark} \label{remaining+} The trajectories entering the switching layer $V_0$ from $V_+$ through $x \in \left(\frac{4n+2}{3}-2\delta,x_{\epsilon,2n+1}^+\right)$ 
will be either attracted to $\mathcal{V}_{2\delta n}$ or expelled again to $V_+$. However, a trajectory moving into $V_+$ with initial condition on $\left(x,v\right)=\left(x_{\epsilon,2n+1}^+ +\alpha,1\right)$, $1\gg\alpha>0$, is upper bounded by the one with initial condition $\left(x,v\right)=\left(x_{\epsilon,2n+1}^+,1\right)$, 
and this one intersects again the line $v=1$ within the (attractive) $x$-interval $\left(x_{\epsilon,2n+2}^+,x_{\epsilon,2n+3}^+\right)$, as shown in \cref{below2}, so the trajectory starting on $\left(x,v\right)=\left(x_{\epsilon,2n+1}^+ +\alpha,1\right)$ will fall therein as well. If such entrance is by  $x \in \left(x_{\epsilon,2n+2}^+,x^+_{2n+3}-2\delta\right)$, then  \cref{lemant} applies and the trajectory is attracted by the corresponding $\mathcal{V}_{2\delta (n+1)}$. Otherwise, the process is iterated, but the 'loss' of energy due to the damping term $-av$ in \cref{nlrb} will eventually take the trajectory to enter $V_0$ by $x \in \left(x_{\epsilon,2(n+k)}^+,x^+_{2(n+k)+2}-2\delta\right)$, and then \cref{lemant} applies.
\end{remark}

An equivalent procedure implies the analogous result to \cref{lemant} for trajectories entering $V_0$ from $V_-$.

\begin{lemma} \label{lemant-} Let $a \in \mathbb{R}^+$ and $1\gg\delta>0$. Then, given $1\gg\hat{\delta}>6\delta$, there exists $\bar{\epsilon}_->0$ such that, $\forall n\geq 1$ and $\epsilon \in \left(0,\bar{\epsilon}_-\right)$, the trajectories of \cref{nlr} entering the switching layer $V_0$ from $V_-$ through $x \in \left(x^-_{\epsilon,2n-1},x^-_{2n}-\hat{\delta}\right)$ reach $\mathcal{V}_{2\delta n}$ in finite time, and keep evolving therein until they leave it by  $x \geq x^-_{2n}-6\delta$.
\qed
\end{lemma}

Let us now compute the distance between the stable critical manifold $\Lambda_{0n}^{N}$ defined in \cref{ssmn} and its associated slow invariant manifold within $\mathcal{V}_{2\delta n}$, which is defined in \cref{s2dn}. The (unique) formal expansion of this slow manifold in any subcompact of $\mathcal{V}_{2\delta n}$ sufficiently far from the fold points $\left(\frac{4n}{3},1\right)$ and $\left(4n,-1\right)$, where $\psi'(v) \neq 0$, is given by
\begin{equation} \label{fm} v(x,\epsilon)=v_0(x)+\epsilon v_1(x)+O\left(\epsilon^2\right),
\end{equation}
with $v_0(x)$ being such that $\left(x,v_0(x)\right)$ stands for the graph of $\Lambda_{0n}^{N}$, i.e.
\begin{align}
 x\left(1+\frac{\varphi(v_0(x))}{2}\right)=2n.
\end{align}
Substituting \cref{fm} into \cref{nlrb} while considering a Taylor expansion for the sinusoidal term in a neighbourhood of $v_0(x)$ and disregarding $O\left(\epsilon^2\right)$ terms yields
\begin{align}
 \dot v_0(x)+\epsilon \dot v_1(x)=-a\left(v_0(x)+\epsilon v_1(x)\right)-\frac{\pi x \varphi'\left(v\right)}{2}.
\end{align}
Then for $\epsilon \rightarrow 0$,
\begin{align}
v_1(x)=2\frac{\dot v_0(x)+av_0(x)}{\pi x \psi'\left(v\right)}.
\end{align}
Let us now consider a compact subset of $\left[\frac{4n}{3}+2\delta,4n-6\delta\right]$ where $v_0(x)$ has a 'flat' shape,  i.e. far enough from $v=\pm 1$, say $\left[\frac{5n}{3},\frac{10n}{3}\right]$. It is immediate that $v_0(x)$, $\dot{v}_0(x)$, and $\psi'(v_0(x))$ are bounded therein, so it follows straightforwardly that
\begin{align}
 v_1(x) \sim O\left(\frac{1}{x}\right) \sim O\left(\frac{1}{n}\right), \ \ \forall x \in \left[\frac{5n}{3},\frac{10n}{3}\right].
\end{align}
Consequently, \cref{fm} implies
\begin{align}
 v(x)-v_0(x) \sim O\left(\frac{\epsilon}{n}\right), \ \ \forall x \in \left[\frac{5n}{3},\frac{10n}{3}\right].
\end{align} 
Then Fenichel theory \cite{f79} and the preceding discussion yield the following Lemma.

\begin{lemma} \label{lfm} For all $n \geq 1$ and $1\gg\delta>0$, the trajectories of \cref{nlr} evolving in the compact set $\mathcal{V}_{2\delta n}$, defined in \cref{s2dn}, are exponentially attracted to a slow invariant manifold that, for all $x \in \left[\frac{7n}{3},3n+2\right]$, is within  $O\left(\frac{\epsilon}{n}\right)$ distance from the stable critical manifold $\Lambda_{0n}^{N}$ defined in \cref{ssmn}. \qed
\end{lemma}

The $O\left(\frac{\epsilon}{n}\right)$-closeness between the stable critical manifold and the slow invariant manifold is lost in a neighbourhood of the fold points, where $\psi'(v)$ approaches zero  and the graph of $v_0(x)$ loses the 'flat' shape. The next result studies the distance between the fold points, i.e. the ideal `exiting points' of the switching layer, and the real exiting points, by turning \cref{nlr} into a Riccati equation. 

\begin{lemma} \label{below1} The trajectories of \cref{nlr} trapped by an attractive slow invariant manifold leave the switching layer by $\left(x_e,-1\right)$, with $x_e \in \left(x^-_{\epsilon,2n},\bar{x}^-_{\epsilon,2n}\right)$, where
\begin{equation} \label{jaquasi} \bar{x}^-_{\epsilon,2n}:=x^-_{\epsilon,2n}+O\left(\left(\epsilon^2/n\right)^\frac{1}{3}\right), \ \ n \in \mathbb{N}. 
\end{equation} 
\end{lemma}

\begin{proof} Using the change of variables 
\begin{align}
 x-x^-_{\epsilon,2n}=\frac{\epsilon^\frac{2}{3}}{\alpha}\tilde{x}, \quad v+1=\frac{\pi \epsilon^\frac{1}{3}}{2 \alpha^2}\tilde{v}, \quad t=\frac{\epsilon^\frac{2}{3}}{\alpha}\tau\;,
\end{align}
with
\begin{align}
 \alpha=\sqrt[3]{\frac{\pi}{2}\arcsin (a\epsilon) + \frac{n\pi^2\psi''\left(-1\right)}{2}}\;,
\end{align} 
which is real because of \cref{psie}, system \cref{nlr} can be written as
\begin{subequations} \label{ric}
\begin{align}
\label{rica} \dot{\tilde{x}}&=1,
\\
\label{ricb} \dot{\tilde{v}}&=-\left(\tilde{x}+{\tilde{v}}^2\right)+O\left(\frac{\epsilon^{2/3}}{n^{1/3}}\tilde{v},\frac{\epsilon^{2/3}}{n^{1/3}}\tilde{x}^2,\frac{\epsilon^{2/3}}{n^{1/3}}\tilde{x}\tilde{v},\epsilon^2\tilde{x},\epsilon^2\tilde{v}\right).
\end{align}
\end{subequations}
Neglecting the perturbative terms, \cref{ric} becomes a Riccati equation, which it is well known (see \cite{rozov80}) to have a unique, decreasing solution
such that
\begin{align} 
\tilde{x}&=-{\tilde{v}}^2+O\left({1}/{\tilde{v}}\right) \ \ \mbox{as} \ \ \tilde{v} \rightarrow +\infty, \\
\tilde{x}&=\tilde{\Omega}_0 \ \ \mbox{as} \ \ \tilde{v} \rightarrow -\infty,
\end{align}
with $\tilde{\Omega}_0 \in \mathbb{R}^+$. This solution guides the deviation of the  Fenichel manifold of the overall system \cref{ric} from the corresponding critical manifold, which is therefore bounded by some $\tilde{\Omega}$ such that $\tilde{\Omega}_0  < \tilde{\Omega}$. Hence, in the original variables we have that
\begin{align}
 x-x^-_{\epsilon,2n} < \sfrac{\epsilon^{{2}/{3}}}{\alpha}\tilde{\Omega} 
\quad\Rightarrow \quad x_e=x(-1)<x^-_{\epsilon,2n}+O\left(\left(\epsilon^2/n\right)^\frac{1}{3}\right)\;.
\end{align}

Notice also that $x_e>x^-_{\epsilon,2n}$, which is consistent with the fact that for all $x \in \left(x^-_{\epsilon,2n-1},x^-_{\epsilon,2n}\right)$ the vector field on $v=-1$ points upwards (i.e. towards $V_0$).

\end{proof}

\begin{remark} \label{ufp}
Although it has not been necessary in the proof, an equivalent procedure would imply that the deviation of the slow invariant manifold from $x^+_{\epsilon,2n}$ on the upper boundary of $V_0$ is also $O\left(\left(\epsilon^2/n\right)^\frac{1}{3}\right)$. \end{remark}

Finally, an analogous situation to that discussed in Remark \ref{remaining+} arises with the trajectories entering $V_-$ which are not guaranteed to reach $\mathcal{V}_{2\delta n}$, i.e., recalling \cref{lemant-}, those entering through  $x \in \left[x^-_{2n}-\bar{\delta},x^-_{\epsilon2n}\right)$. An equivalent reasoning implies that they will eventually catch up with a stable slow invariant manifold further on, and in any case never reach $V_+$. 

\cref{fig:trajectories} illustrates the above analysis. Trajectories are in red, while vector field directions are in blue.

\begin{figure}[t]\centering
\includegraphics[width=0.98\textwidth]{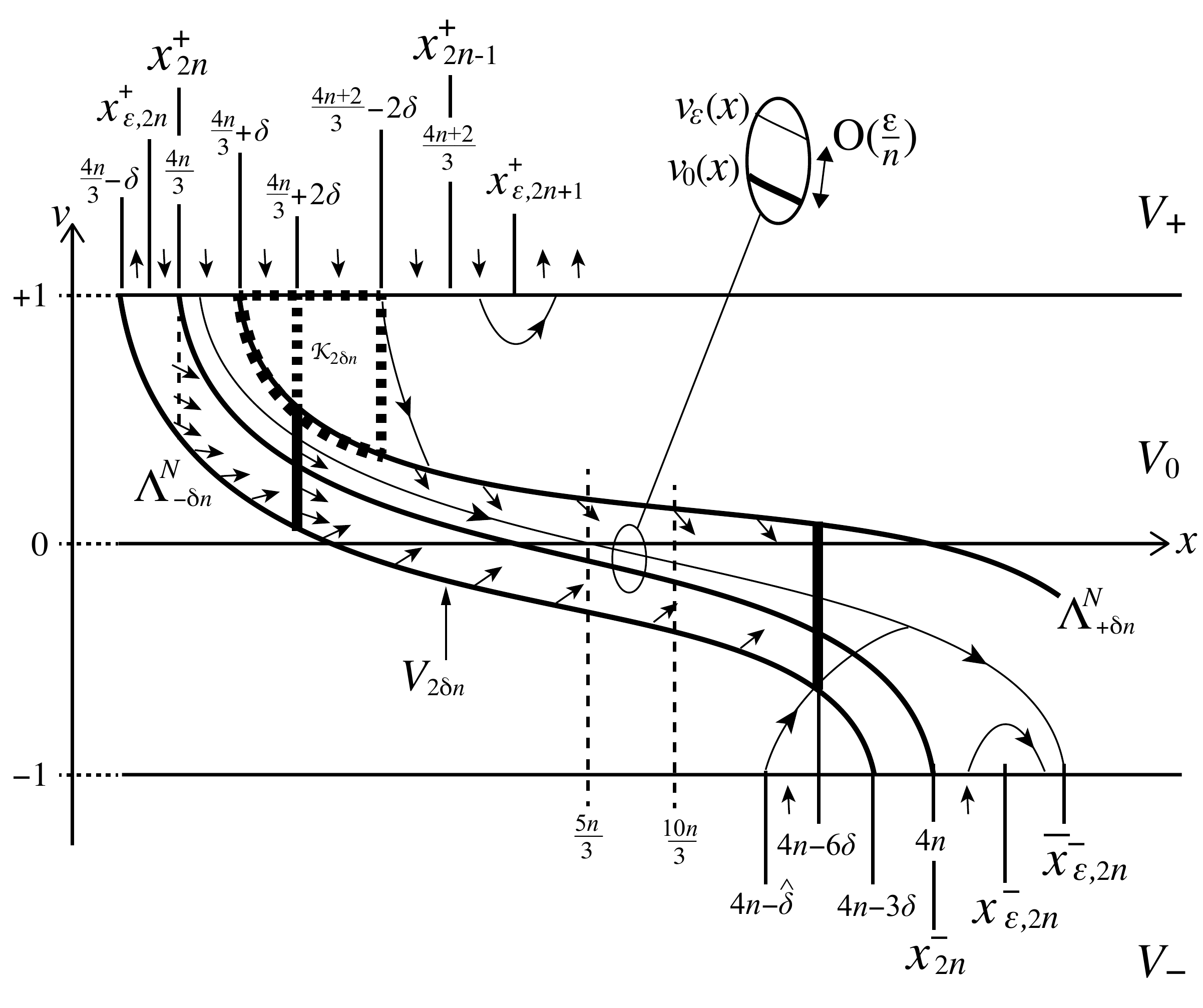} 
\caption{Nonlinear regularization: dynamical analysis as described in the text. (Not to scale)}\label{fig:trajectories}
\end{figure} 

Recall from \cref{nls} that in the nonlinear switching system, solutions with initial conditions in $S_+$ eventually find themselves constrained to $S_-\cup S_0$. It follows straightforwardly from the discussion above that this follows similarly for the regularized system.

\begin{proposition} \label{totavall} For any solution of \cref{nlr} there exists $x_T \in \mathbb{R}_{\geq 0}$ such that, for all $x \geq x_T$, its evolution is constrained in $V_0 \cup V_-$. \qed
\end{proposition}

Since the critical manifolds change qualitatively with $x$ inside the switching layer, the existence of periodic solutions is impossible. 

\begin{theorem} \label{lasteo} System \cref{nlr} does not possess periodic solutions. \qed
\end{theorem}

Despite this the system does possess solutions tending asymptotically towards periodic functions; these functions are only exact solutions for the limiting ($\eps=0$) system. More precisely we show below that, for $x \rightarrow +\infty$, all trajectories tend to a periodic object which, when $\epsilon \rightarrow 0$, tends to the periodic solution of the discontinuous system given by \cref{nlps}. 

It follows from the previous discussion that every trajectory is eventually trapped by a slow invariant manifold within the switching layer and exits into $V_-$ at exponentially close distance to $\left(\bar{x}^-_{\epsilon,2n},-1\right)$. Moreover, according to \cref{below2} and standard results on regularity of initial conditions and parameters, when $x>x^-_{\epsilon,2n+1}$ the trajectory finds the vector field in $V_-$ pointing back towards $V_0$, so it is mapped again onto the switching layer at exponentially close distance from the entry point of the trajectory with initial condition $\left(\bar{x}^-_{\epsilon,2n},-1\right)$, namely, 
\begin{align}
 x_e:=\bar{x}^-_{\epsilon,2n}+x_{\epsilon,a}-O\left(\left(\epsilon^2/n\right)^\frac{1}{3}\right) \in \left(x^-_{\epsilon,2n+1},x^-_{\epsilon,2n+2}\right),
\end{align} 
where, according to \cref{foldr},
\begin{align}
x_{\epsilon,a}=x(\epsilon,a) \in \left(2-2\arcsin (a \epsilon),4\right).
\end{align}

If the entry point on $V_0$ belongs to the interval defined in \cref{lemant-} then the trajectory is trapped by the next attractive slow  invariant manifold and leaves $V_0$ at exponentially close distance of  $\left(\bar{x}^-_{\epsilon,2n+2},-1\right)$, alternatively it can be either trapped by such manifold and leave $V_0$ as described or be expelled onto $V_-$ through $\left(x,-1\right)$, with $x \in \left(x^-_{\epsilon,2n+2},\bar{x}^-_{\epsilon,2n+2}\right)$. After this, the process is repeated iteratively. 

Notice also that the value of $\psi(v)$ at $x_e$  for the   $2n+2$ critical manifold is such that
\begin{align}&  \left. \psi \left(v\right) \right|_{x=x_e} =2\left(\frac{2(n+1)}{4n+\arcsin (a\epsilon)-O\left(\left(\epsilon^2/n\right)^\frac{1}{3}\right)+x_{\epsilon,a}}-1\right) \xrightarrow[n \rightarrow +\infty]{}  -1 \nonumber\\ & \Rightarrow \quad \left. v \right|_{x=x_e} \rightarrow -1.
\end{align}
As the corresponding exponentially attractive slow invariant manifold lies at $O\left(\epsilon/n\right)$ distance from the critical manifold it is immediate that, for $n \rightarrow +\infty$, the piece of trajectory going from an entry point on $V_0$ until the next exit point tends to $-1$.   

Notice, however, that the trajectory with initial condition $\left(\bar{x}^-_{\epsilon,2n},-1\right)$ is lower bounded by that with initial condition $\left(x^-_{\epsilon,2n},-1\right)$, the mutual distance being of $O\left(\left(\epsilon^2/n\right)^\frac{1}{3}\right)$ because of standard regularity results. It is immediate that the entry point of such trajectory on $V_0$ is by $\left(x^-_{\epsilon,2n}+x_{\epsilon,a},-1\right)$, and an analogous analysis of $\psi(v)$ reveals that, for $n \rightarrow +\infty$, 
\begin{align}
 \left. v \right|_{x=x^-_{\epsilon,2n}+x_{\epsilon,a}} \rightarrow -1 
\end{align}
as well, and so happens with the piece of trajectory in $\left(x^-_{\epsilon,2n}+x_{\epsilon,a},x^-_{\epsilon,2n+2}\right)$.

\begin{definition} \label{vr} Let $v_r(x)$ be the $4$-periodic function defined as follows 
\begin{align} 
& v_r(x)=v_-\left(x,x^-_{\epsilon,2n}\right), \ \ x \in \left[x^-_{\epsilon,2n},x^-_{\epsilon,2n}+x_{\epsilon,a}\right], \\ 
& v_r(x)=-1, \ \ x \in \left(x^-_{\epsilon,2n}+x_{\epsilon,a},x^-_{\epsilon,2n+2}\right),
\end{align}
with $v_-\left(x,x^-_{\epsilon,2n}\right)$ denoting the solution of \cref{nlr} with initial condition $\left(x^-_{\epsilon,2n},-1\right)$ introduced in \cref{below2}.
\end{definition}



\begin{theorem} \label{last} As $x \rightarrow +\infty$, all trajectories of system \cref{nlr} tend to the 4-periodic function $v_r(x)$ of \cref{vr}. 
\qed
\end{theorem}

\begin{remark} \label{rempo}  Notice that, for $\epsilon \rightarrow 0$, 
\begin{align}
 x^-_{\epsilon,2n},\bar{x}^-_{\epsilon,2n} \rightarrow x^-_{2n}=4n \ \ \mbox{and} \ \ x_{\epsilon,a}\in \left(2-2\arcsin (a\epsilon),4\right) \rightarrow (2,4).
\end{align} Hence, the function $y_r(x)$ obtained by reversing the change of variables $v=\frac{y}{\epsilon}$ is such that $y_r \rightarrow y_d$ for $\epsilon \rightarrow 0$, $y_d$ being the sliding 4-periodic solution of the discontinuous nonlinear system introduced in \cref{nlps}.
\end{remark}

\begin{figure}[ht!]\centering
\includegraphics[width=0.8\textwidth]{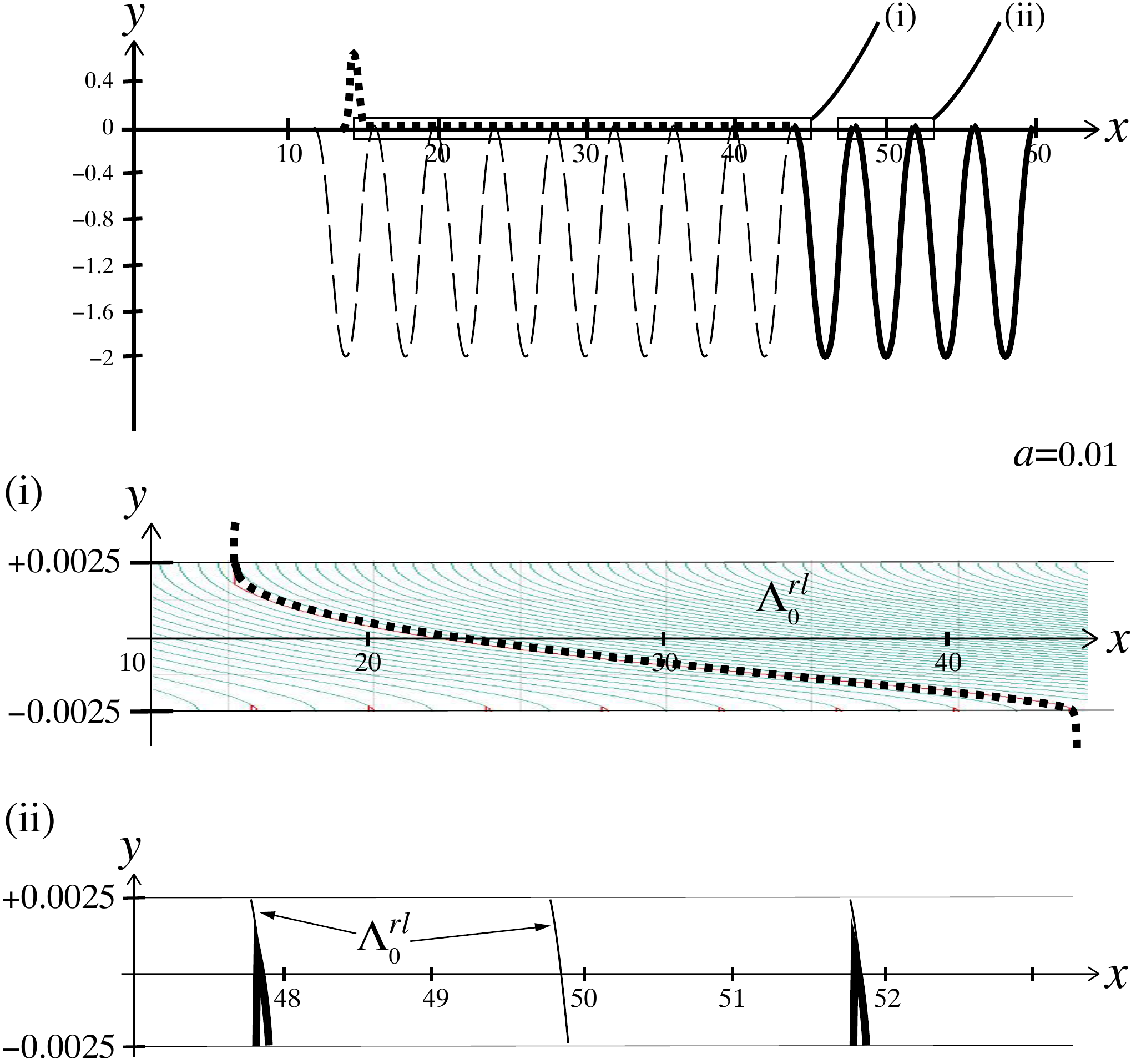}
\caption{Nonlinear regularization, showing the period 4 sliding orbit (full curve), and two trajectories (dotted and dashed) that collapse onto it. Simulated from \cref{nlr} with $a=0.01$ and $\epsilon=0.0025$. Two magnifications of the switching layer are shown: (i) the critical manifolds $\Lambda_0^{L}$ and the dotted trajectory evolving close to one branch, (ii) the small sliding segment of the period 4 sliding orbit. The periodic orbit pass through $\left(x_0,v_0\right)=\left(14.1,1.1\right)$.}\label{fig:f2_z}
\end{figure} 

In \cref{fig:f2_z} we simulate a pair of solutions of \cref{nlr} with $a=0.01$ and $\epsilon=0.0025$, starting at $\left(x_0,v_0\right)=\left(14.1,1.1\right)$ and $\left(x_0,v_0\right)=\left(12.1,-1.1\right)$, both of which collapse toward the period 4 solution. The former does so via a long segment of sliding along a slow invariant manifold, close to the critical manifold $\Lambda_0^{N}$ branch with $n=22$, exiting the layer at $x_e \approx 2\cdot 22=44$ after sliding for a distance $\Delta x=29.9$, more than 7 of the periods portrayed in \cref{fig:sign}, during which it passes through crossing, repelling, and attractive sliding regions ($S_{0C}$, $S_{0R}$, and $S_{0A}$). The latter solution lies close throughout to the periodic solution, but remains in $S_-$ and also only starts sliding after $x\approx44$.  The behaviours described in \cref{totavall} and \cref{last} are thus illustrated.

\section{Conclusions} \label{sec:conc}

In \cref{sec:idl} we showed that the linear switching system has a period 4 orbit, which is a crossing orbit for $a<a_l\ll1$ and involves sliding for $a>a_h\gg1$. We expect that this orbit exists for all $a$ (i.e. also for $a\ge a_l$ but not large), presumably undergoing a grazing-sliding bifurcation at some value in the range $a_l<a<a_h$. In the nonlinear system we showed that there exists a period 4 sliding solution for all $a$. The forms of the periodic orbits in the linear and nonlinear systems bare little relation to each other. When the discontinuity was regularized the periodicity of the linear system persisted, while the nonlinear system tended asymptotically towards a periodic orbit that was no longer strictly a solution. 

In the context of the resistor-inductor circuit in \cref{fig:RL} these results manifest as the current dynamics in cases where the controller can be modelled as switching linearly or non-linearly between frequency states. In the linear case the circuit finds a stable oscillation, in the nonlinear case it ages until it becomes constrained to lower frequency state and the switching threshold. Experiments to observe such behaviour would certainly be of future interest, in electronic controllers or in other possible applications, with the focus being on distinguishing the linear-versus-nonlinear models of switching and the phenomenon of ageing. 

Looking forward towards study of the original motivating system \cref{3sys}, we expect that the distinctions between linear and nonlinear switching arise very similarly to the relatively simple dynamics of the first order oscillator studied here. The added complexity of the second order oscillator lies mainly in the added dimension creating much richer sliding dynamics: whereas in the first order oscillator the sliding manifolds act merely as a barrier over long times, in the second order oscillator the ageing sliding manifolds bring to life a sequence of changing behaviours. Our preliminary investigations show first that lengthier segments of sliding are seen, similar to the first order oscillator here, followed later (for larger $t$) by relaxation oscillations and mixed-mode oscillations. These remain to be studied in detail, but with the first order system we have singled out the most important element that the example in \cite{j15hidden} was conceived to show, namely the importance of linear versus nonlinear terms in models of switching, the different dynamics they create, and their persistence under regularization. 


\section*{Acknowledgments}

J.M.O. was partially supported by the Government of Spain through the \emph{Agencia Estatal de Investigaci\'on} Project DPI2017-85404-P and by the Generalitat de Catalunya through the \emph{AGAUR} Project 2017 SGR 872. 
C.B. was partially supported by the Spanish government grant PGC2018-098676-B-I00.
P.M. was partially supported by the Spanish government grant PGC2018-100928-B-I00.
C.B. and P.M. were partially supported by the Catalan government grant 2017-SGR-1049. 

\bibliography{grazcat}
\bibliographystyle{plain}

\end{document}